\documentclass[11pt,reqno,british,twoside]{article}

\usepackage[utf8]{inputenc}
\usepackage[sc]{mathpazo}
\usepackage{babel,mathtools,amsmath,amssymb,url,paralist,xspace,mathrsfs}
\usepackage{booktabs,microtype,array,fixltx2e,fancyhdr,amscd}
\usepackage[margin=30pt,font=small,labelfont=bf,labelsep=period]{caption}
\usepackage[today,nofancy]{rcsinfo}
\usepackage[amsmath,thmmarks]{ntheorem}
\rmfamily
\mathtoolsset{mathic=true}
\setdefaultenum{\upshape (i)}{}{}{}

\makeatletter
\DeclareRobustCommand*{\bfseries}{%
  \not@math@alphabet\bfseries\mathbf
  \fontseries\bfdefault\selectfont
  \boldmath
}
\makeatother

\theoremstyle{plain}
\theoremseparator{.}
\newtheorem{theorem}{Theorem}[section]
\newtheorem{proposition}[theorem]{Proposition}
\newtheorem{lemma}[theorem]{Lemma}
\newtheorem{corollary}[theorem]{Corollary}

\theorembodyfont{\normalfont}
\newtheorem{definition}[theorem]{Definition}

\theoremheaderfont{\itshape}
\theoremsymbol{\begin{small}\ensuremath{\quad\triangle}\end{small}}
\newtheorem{remark}[theorem]{Remark}
\theoremsymbol{\begin{small}\ensuremath{\quad\diamondsuit}\end{small}}
\newtheorem{example}[theorem]{Example}

\theoremstyle{nonumberplain}
\theoremheaderfont{\itshape}
\theorembodyfont{\normalfont}
\theoremsymbol{\begin{small}\ensuremath{\quad\Box}\end{small}}
\newtheorem{proof}{Proof}

\qedsymbol{\begin{small}\ensuremath{\quad\Box}\end{small}}

\numberwithin{equation}{section}
\numberwithin{table}{section}


\newcolumntype{C}{>{$}c<{$}}
\newcolumntype{L}{>{$}l<{$}}
\newcolumntype{R}{>{$}r<{$}}

\let\oldbibliography\thebibliography
\renewcommand{\thebibliography}[1]{%
  \oldbibliography{#1}%
  \small
  \setlength{\itemsep}{0pt}%
  \setlength{\parskip}{0pt plus 1pt}%
}

\newcommand{\Lie}[1]{\operatorname{\textsl{#1}}}
\newcommand{\lie}[1]{\operatorname{\mathfrak{#1}}}
\newcommand{\Lel}[1]{{\mathsf{#1}}}  
\newcommand{\GL}{\Lie{GL}}
\newcommand{\SO}{\Lie{SO}}
\newcommand{\Sp}{\Lie{Sp}}
\newcommand{\SU}{\Lie{SU}}
\newcommand{\Un}{\Lie{U}}
\newcommand{\Spin}{\Lie{Spin}}
\newcommand{\g}{\lie{g}}
\newcommand{\h}{\lie{h}}
\newcommand{\lk}{\lie{k}}
\newcommand{\lr}{\lie{r}}
\newcommand{\ld}{\lie{d}}
\newcommand{\sP}{\lie{sp}}
\newcommand{\su}{\lie{su}}
\newcommand{\lt}{\lie{t}}
\newcommand{\un}{\lie{u}}
\newcommand{\Ld}{\mathcal L}
\newcommand{\Pg}[1][\g]{\mathcal P_{#1}}
\newcommand{\dP}{d_{\mathcal P}}
\newcommand{\X}{\mathcal X}
\newcommand{\Xt}[1][t]{\X_{#1}}
\newcommand{\CST}{\mathscr C}

\newcommand{\Hodge}{{*}}

\newcommand{\bC}{{\mathbb C}}
\newcommand{\bH}{{\mathbb H}}
\newcommand{\bR}{{\mathbb R}}

\DeclareMathOperator{\ad}{ad}
\DeclareMathOperator{\Ad}{Ad}
\DeclareMathOperator{\CP}{\bC P}
\DeclareMathOperator{\Der}{Der}
\DeclareMathOperator{\Hom}{Hom}
\DeclareMathOperator{\img}{im}
\DeclareMathOperator{\im}{Im}
\DeclareMathOperator{\re}{Re}
\DeclareMathOperator{\Tr}{Tr}
\DeclareMathOperator{\diag}{diag}
\DeclareMathOperator{\stab}{stab}
\DeclareMathOperator{\vol}{vol}

\newcommand{\hook}{{\lrcorner\,}}

\newcommand{\HKT}{\textsc{HKT}\xspace}
\newcommand{\NB}{\nabla}
\newcommand{\LC}{\NB^{\textup{LC}}}

\DeclarePairedDelimiter{\norm}{\lVert}{\rVert}
\DeclarePairedDelimiter{\Span}{\langle}{\rangle}

\DeclarePairedDelimiter{\inpd}{\langle}{\rangle}
\newcommand{\inp}[2]{\inpd{#1,#2}}

\renewcommand{\bfdefault}{b}

\newcommand{\br}{\hspace{0pt}}
\newcommand{\bdash}{-\br} 
\newcommand{\eqbreak}{\\&\qquad}

\begin{document}

\thispagestyle{empty}
\begin{small}
  \begin{flushright}
    IMADA Preprint 2010\\
    CP\textsuperscript3-ORIGINS: 2010-53
  \end{flushright}
\end{small}

\bigskip

\begin{center}
  \LARGE\bfseries Multi-moment maps
\end{center}
\begin{center}
  \Large Thomas Bruun Madsen and Andrew Swann
\end{center}

\begin{abstract}
  We introduce a notion of moment map adapted to actions of Lie groups
  that preserve a closed three-form.  We show existence of our
  multi-moment maps in many circumstances, including mild topological
  assumptions on the underlying manifold.  Such maps are also shown to
  exist for all groups whose second and third Lie algebra Betti
  numbers are zero.  We show that these form a special class of
  solvable Lie groups and provide a structural characterisation.  We
  provide many examples of multi-moment maps for different geometries
  and use them to describe manifolds with holonomy contained in \( G_2
  \) preserved by a two-torus symmetry in terms of tri-symplectic
  geometry of four-manifolds.
\end{abstract}

\bigskip
\begin{center}
  \begin{minipage}{0.8\linewidth}
    \begin{small}
      \tableofcontents
    \end{small}
  \end{minipage}
\end{center}

\bigskip
\begin{small}\noindent
\end{small}
\newpage

\section{Introduction}
\label{sec:introduction}

One illuminating example of the interplay between mathematics and
physics is the relation between symplectic geometry and mechanics.  A
symplectic manifold is characterised by a closed, non-degenerate form
of degree two.  In modern physics higher degree forms play an
important role too.  While some authors have looked at extensions of
field theories, closed three-forms appear to be particularly relevant
in supersymmetric theories with Wess-Zumino terms, string theory and
one-dimensional quantum mechanics
\cite{Michelson-S:conformal,Strominger:superstrings,Gates-HR:twisted,Baez-HR:string}.
They have been studied mathematically in a number of contexts
including stable forms \cite{Hitchin:forms}, strong geometries with
torsion \cite{Fino-PS:SKT}, gerbes \cite{Brylinski:loop} and
generalised geometry
\cite{Hitchin:generalized-CY,Gualtieri:generalized}.

One construction illustrating the link between symplectic geometry and
physics is that of moment maps.  A moment map is an equivariant map
from a symplectic manifold into the dual of the Lie algebra of a Lie
group acting by symplectomorphisms.  It captures the concepts of
linear and angular momentum from mechanics.  The main purpose of this
paper is to explain that a similar type of map exists when we are
given a manifold \( M \) with a closed three-form \( c \) and a Lie
group \( G \) that acts on \( M \) preserving \( c \).  We shall call
the pair \( (M,c) \) a \emph{strong geometry}, and we refer to the Lie
group \( G \) as a \emph{group of symmetries}.  We write \( \lie g \)
for the Lie algebra of \( G \).

An important feature of our construction is that the resulting
multi-moment map is a map from \( M \) to a vector subspace~\( \Pg^*
\) of \( \Lambda^2\lie g^* \), with \( \Pg^* \) independent of~\( M
\).  This is in contrast to previous considerations
\cite{Carinena-CI:multisymplectic,Gotay-IMM:covariant} of so-called
covariant moment maps \( \sigma \colon M \to \Omega^1(M,\lie g^*) \),
which are defined via the relation
\begin{equation}
  \label{eq:covariant}
  d\inp \sigma {\Lel X} = X \hook c,\qquad\text{for all \(
    \Lel X \in \lie g \)},
\end{equation}
where \( X \) is the vector field on \( M \) generated by \( \Lel X
\in \lie g \).  Here the target space \( \Omega^1(M,\lie g^*) \) is an
infinite-dimensional space depending both on \( M \) and on \( \lie g
\).  We also note that finding covariant moment maps can be hard;
equation~\eqref{eq:covariant} has a solution \( \inp\sigma{\Lel X} \)
only if the cohomology class \( [X\hook c] \) vanishes in \( H^2(M)
\).  Thus, existence of covariant moment maps often requires some
non-trivial topological assumption such as \( b_2(M) = 0 \).

In contrast, we will show that multi-moment maps exist under mild
topological assumptions: if \( M \) is simply-connected and either \(
G \) is compact or \( M \) is compact with \( G \)-invariant volume
form.  This is analogous to symplectic moment maps, and enables us to
give many examples.  As one application, we will use multi-moment maps
to study seven\bdash manifolds with holonomy contained in \( G_2 \),
when these have a free isometric action of a two-torus.  We find that
the geometry is determined by a conformal structure on a four-manifold
specified by a certain triple of symplectic two-forms.  This extends
the work of Apostolov \& Salamon \cite{Apostolov-S:K-G2} and fits with
the perspective of Donaldson \cite{Donaldson:two-forms}.

In the symplectic case, there is also a general existence theorem for
moment maps in the case that the symmetry group is semi-simple; it is
a result that does not require any topological assumptions on the
manifold.  Note that semi-simplicity of a Lie group is characterised
algebraically by the vanishing of the first and second Betti numbers
of the Lie algebra cohomology.  In this direction, we prove that
multi-moment maps exist whenever the second and third Betti numbers \(
b_2(\lie g) \) and \( b_3(\lie g) \) of the Lie algebra cohomology of
\( G \) vanish.  We call Lie algebras of this type \emph{\( (2,3)
  \)-trivial}.  The weaker setting of Lie algebras with \( b_2(\lie g)
= 0 \), where multi-moment maps are unique if defined, provides many
examples of homogeneous strong geometries, including examples that are
2-plectic in the terminology of \cite{Baez-HR:string}.

As far as we know, \( (2,3) \)-trivial algebras have not been studied
before.  We show that these are solvable Lie algebras, that are not
products of smaller dimensional algebras.  Their derived algebra is of
codimension one, and is necessarily nilpotent.  From this one may
classify the low-dimensional examples, and further study leads to a
characterisation of the allowed solvable extensions of nilpotent
algebras.  The structure theory shows that many examples exist,
including some that are unimodular.  On the other hand one finds that
some nilpotent algebras can not be realised as the derived algebra of
a \( (2,3) \)-trivial algebra.

This paper is organised as follows.  In section~\ref{sec:definition}
we give the fundamental calculations that lead to the definition of
multi-moment map and introduce the Lie kernel \( \Pg \) of a Lie
algebra~\( \lie g \).  We then consider topological and algebraic
criteria for existence and uniqueness of multi-moment maps in
section~\ref{sec:exc}.  As discussed above \( (2,3) \)-trivial Lie
algebras play a natural role and section~\ref{sec:solalg} is devoted
to an algebraic study of this class and the description of a number of
examples.  We then return to strong geometries and their multi-moment
maps.  The basic example is provided by the total space \(
\Lambda^2T^*N \) of the second exterior power of the cotangent bundle
of a manifold \( N \).  Homogeneous strong geometries with
multi-moment maps are closely tied to orbits in the dual \( \Pg^* \)
of the Lie kernel and we develop a Kirillov-Kostant-Souriau type
theory, pointing out links with nearly Kähler and hypercomplex
geometry.  The final section of the paper is devoted to an
investigation of torsion-free \( G_2 \)-manifolds with an isometric
action of a two-torus.  We show how multi-moment maps lead to a
description of such metrics via tri-sympletic geometry of
four-manifolds.

Some of the algebraic material of this paper is supplemented by our
work in \cite{Madsen-S:2-3-trivial}.  Future work will address
extensions of the final section providing multi-moment map approaches
to torsion-free \( \Spin(7) \)-structures with \( T^3 \)-symmetry.

\paragraph*{Acknowledgements}
We gratefully acknowledge financial support from \textsc{ctqm},
\textsc{geomaps} and \textsc{opalgtopgeo}.  AFS is also partially
supported by the Ministry of Science and Innovation, Spain, under
Project \textsc{mtm}{\small 2008-01386} and thanks \textsc{nordita}
for hospitality.

\section{Main definitions}
\label{sec:definition}

Let \( (M,c) \) be a strong geometry, meaning that \( M \) is a smooth
manifold and that \( c \) is a closed three-form on \( M \).  Note
that unlike the symplectic case there is no one canonical form for \(
c \), not even pointwise on~\( M \).  In general, we do not require
any non-degeneracy of~\( c \).  However, when necessary we will use
the terminology of \cite{Baez-HR:string} that \( c \) is
\emph{2-plectic} if \( X\hook c=0 \) at \( x\in M \) only when \( X=0
\) in \( T_xM \).

\begin{remark}
  Since \( c \) is closed, \( \ker c = \{\,X\in TM:X\hook c=0\,\} \)
  is integrable.  Thus if \( \ker c \) is of constant rank and has
  closed leaves, \( c \) induces a 2-plectic structure on \( M/\ker c
  \).
\end{remark}

\begin{remark}
  \label{rem:strongndeg}
  One could consider strongly non-degenerate three-forms \( c \),
  meaning that \( c(X,Y,\cdot) \ne 0 \) for all \( X\wedge Y\ne 0 \).
  However, by \cite{Massey:cross-products} such \( c \) exist only in
  dimensions \( 3 \) and~\( 7 \).  The former case is given by a
  volume form, the latter by a \( G \)-structure with \( G = G_2 \) or
  its non-compact dual.
\end{remark}

Let \( G \) be a group of symmetries for \( (M,c) \), meaning that \(
G \) acts on \( M \) preserving the three-form \( c \).  Thus for each
\( \Lel X \in \lie g \) we have \( \Ld_Xc=0 \), where \( X \) is the
vector field generated by \( \Lel X \).  As \( dc=0 \), this gives
\begin{equation}
  \label{eq:X-c}
  0 = \Ld_Xc = d(X\hook c) + X\hook dc = d(X\hook c),
\end{equation}
so the two-form \( X\hook c \) is closed.  Suppose \( \Lel Y\in \lie g
\) commutes with \( \Lel X \).  Then we have
\begin{equation*}
  0 = \Ld_Y(X\hook c) =  d(Y\hook X\hook c) = d((X\wedge Y) \hook c),
\end{equation*}
showing that the one form \( (X\wedge Y)\hook c = c(X,Y,\cdot) \) is
closed.  If for example, \( b_1(M) = 0 \), we may then write
\begin{equation*}
  (X\wedge Y) \hook c = d \nu_{X\wedge Y}
\end{equation*}
for some smooth function \( \nu_{X\wedge Y}\colon M\to \bR \).  This
is the basis of the construction of the multi-moment map.  However,
the set of decomposable elements \( \Lel X\wedge \Lel Y \) in \(
\Lambda^2\lie g \) for which \( \Lel X \) and \( \Lel Y \) commute is
a complicated variety.  It is more natural to consider the following
submodule of \( \Lambda^2\lie g \).

\begin{definition}
  \label{def:Liekernel}
  The \emph{Lie kernel} \( \Pg \) of a Lie algebra \( \g \) is the \(
  \g \)-module
  \begin{equation*}
    \Pg:=\ker\left(L\colon \Lambda^2\g\to\g\right),
  \end{equation*}
  where \( L \) is the linear map induced by the Lie bracket.
\end{definition}

The previous calculation may now be extended to elements of the Lie
kernel.  For a bivector \( p = \sum_{j=1}^k X_j\wedge Y_j \) we write
\begin{equation*}
  p \hook c := \sum_{j=1}^k c(X_j,Y_j,\cdot).
\end{equation*}

\begin{lemma}
  \label{lem:kercalc}
  Suppose \( G \) is a group of symmetries of a strong geometry \(
  (M,c) \).  Let \( \Lel p=\sum_{j=1}^k\Lel X_j\wedge \Lel Y_j \) be
  an element of the Lie kernel \( \Pg \) and let \( p = \sum_{j=1}^k
  X_j\wedge Y_j \) be the corresponding bivector on~\( M \).  Then
  \begin{equation}
    \label{eq:kercalc}
    d(p \hook c) = 0.
  \end{equation}
\end{lemma}

\begin{proof}
  The condition that \( \Lel p \) lies in \( \Pg \) is that \(
  0=L(\Lel p)=\sum_{j=1}^k [\Lel X_j,\Lel Y_j] \).  This together with
  \eqref{eq:X-c} and \( dc=0 \) gives
  \begin{equation*}
    \begin{split}
      0&=\sum_{j=1}^k[Y_j,X_j] \hook c = \sum_{j=1}^k
      \left([\Ld_{Y_j},X_j\hook]c\right)\\
      &=\sum_{j=1}^k d(Y_j\hook X_j\hook c) + Y_j \hook d(X_j\hook c)
      - X_j\hook d(Y_j \hook c) - X_j \hook Y_j \hook dc \\
      &=\sum_{j=1}^k d(Y_j\hook X_j\hook c) = d (p\hook c),
    \end{split}
  \end{equation*}
  as required.
\end{proof}

Thus if for example \( b_1(M) = 0 \), there is a smooth function \(
\nu_p\colon M\to \bR \) with \( d\nu_p = p\hook c \) for each \( \Lel
p \in \Pg \).

We are now able to define the main object to be studied in this paper.

\begin{definition}
  \label{def:mmmap}
  Let \( (M,c) \) be a strong geometry with a symmetry group \( G \).
  A \emph{multi-moment map} is an equivariant map \( \nu\colon
  M\to\Pg^* \) satisfying
  \begin{equation}
    \label{eq:mmmapdefrel}
    d\inp \nu{\Lel p} = p \hook c
  \end{equation}
  for each \( \Lel p \in \Pg \).
\end{definition}

Note that for \( G \) Abelian \( \Pg = \Lambda^2\lie g \).  On the
other hand if \( G \) is a compact simple Lie group then the Lie
kernel is a module familiar from a special class of Einstein
manifolds.  Indeed Wolf \cite[Corollary 10.2]{Wolf:isotropy}
(cf. \cite[Proposition 7.49]{Besse:Einstein}) showed that in this case
\( \Lambda^2\lie g = \g \oplus \Pg \) as a sum of irreducible modules,
so \( \SO(\dim G)/G \) is an isotropy irreducible space.

\section{Existence and uniqueness}
\label{sec:exc}

As mentioned in the introduction, one of the principal advantages of
multi-moment maps over covariant moment maps is that one can prove
that multi-moment maps are guaranteed to exist under a wide range of
circumstances.

We start first with topological criteria.
 
\begin{theorem}
  \label{thm:mmpexc}
  Let \( (M,c) \) be a strong geometry with a symmetry group \( G \)
  and assume that \( b_1(M)=0 \).  If either
  \begin{compactenum}
  \item \( G \) is compact, or
  \item \( M \) is compact and orientable, and \( G \) preserves a
    volume form on \( M \),
  \end{compactenum}
  then there exists a multi-moment map \( \nu\colon M\to\Pg^* \).
\end{theorem}

\begin{proof}
  Working component by component, we may assume that \( M \) is
  connected.  As noted after Lemma~\ref{lem:kercalc} the condition \(
  b_1(M) = 0 \) ensures that there are functions \( \nu_{\Lel p} \)
  with \( d\nu_p = p\hook c \) for each \( \Lel p \in \Pg \).
  However, each of these functions may be adjusted by adding a real
  constant.  To build a multi-moment map \( \nu \) via \( \inp\nu{\Lel
    p} = \nu_p \) we need to ensure equivariance.  In the two cases
  above this may be achieved by either averaging over \( G \) or over
  \( M \).  In the second case, one chooses \( \nu_p \) with mean
  value \( 0 \).  In the first case, one chooses a basis \( (\Lel p_i)
  \) of \( \Pg \) and puts \( \nu(m) = \int_G \sum_i
  \Ad_g^*(\nu_{p_i}(g\cdot m)) \vol_G \).  In both cases
  equation~\eqref{eq:mmmapdefrel} is satisfied, and \( \nu \) is a
  multi-moment map.
\end{proof}

As we saw in the above proof, one crucial point is making a canonical
choice of function \( \nu_p \).  The following situation occurs in
many examples and provides a differential geometric criterion for a
construction of multi-moment maps.

\begin{proposition}
  \label{prop:exact}
  Suppose \( G \) is a group of symmetries of a strong geometry \(
  (M,c) \) and that there exists a \( G \)-invariant \( 2 \)-form \(
  b\in\Omega^2(M) \) such that \( db = c \).  Then \( \nu\colon M\to
  \Pg^* \) given by
  \begin{equation}
    \label{eq:exmmmap}
    \inp \nu{\Lel p} = b(p)
  \end{equation}
  is a multi-moment map.
\end{proposition}

\begin{proof}
  The map \( \nu \) is equivariant, since \( b \) is invariant.  We
  have \( \nu_p = b(p) \) with \( d(b(p)) = d (p\hook b) = p\hook dc
  \) by the calculation in Lemma~\ref{lem:kercalc}, so
  equation~\eqref{eq:mmmapdefrel} is satisfied, as required.
\end{proof}

Let us now turn to algebraic criteria for multi-moment maps.  This
involves study of the Lie kernel.  The dual of the exact sequence
\begin{equation*}
  \begin{CD}
    0 @>>> \Pg @>\iota>> \Lambda^2\g @>L>> \g
  \end{CD}
\end{equation*}
is the sequence
\begin{equation}
  \label{eq:exseq1}
  \begin{CD}
    \g^* @>d>> \Lambda^2\g^* @>\pi>> \Pg^* @>>> 0,
  \end{CD}
\end{equation}
which is also exact.  Hence the dual~\( \Pg^* \) of the Lie kernel can
be identified with the quotient space \( \Lambda^2\g^* / d(\g^*) \).
As \( B^1(\g) = d(\g^*) \) is a subspace of \( Z^2(\g) = \ker(d \colon
\Lambda^2\g^* \to \Lambda^3\g^*) \), we have an induced linear map
\begin{equation*}
  \dP\colon\Pg^* \to \Lambda^3\g^*.
\end{equation*}
More concretely given \( \beta \in \Pg^* \), we choose \( \tilde\beta
\in \pi^{-1}(\beta) \) and then \( \dP\beta = d\tilde\beta \).

Let \(b_n(\g)\) denote the dimension of the \( n \)th Lie algebra
cohomology group, so \( b_n(\g) = \dim H^n(\g) = \dim Z^n(\g) - \dim
B^n(\g) \).  The next result follows directly from the above
discussion.

\begin{proposition}
  \label{prop:dmap}
  The linear map \( \dP \colon \Pg^* \to \Lambda^3\g^* \) is a \( \g
  \)-morphism with image contained in \( Z^3(\g) \).  It is injective
  if and only if \( b_2(\g)=0 \).  If this condition holds then \( \dP
  \) is an isomorphism from \( \Pg^* \) onto \( Z^3(\g) \) if and only
  if \( b_3(\g)=0 \).  \qed
\end{proposition}

We will see that this distinguishes a class of Lie groups and Lie
algebras that play a special role in the theory of multi-moment maps
analogous to the role of semi-simple groups in symplectic geometry.
We therefore make a definition.

\begin{definition}
  \label{def:23triv}
  A connected Lie group \( G \) or its Lie algebra \( \g \) that
  satisfies \( b_2(\g) = 0 = b_3(\g) \) will be called
  \emph{(cohomologically) \((2,3)\)-trivial}.
\end{definition}

\begin{theorem}
  \label{thm:algmmpexc}
  Let \( (M,c) \) be a strong geometry with connected
  \((2,3)\)-trivial symmetry group \( G \) acting nearly effectively.
  Then there exists a unique multi-moment map \( \nu\colon M\to\Pg^*
  \).

  More generally, if just \( b_2(\g) = 0 \), then multi-moment maps
  for nearly effective actions of \( G \) are unique when they exist.
\end{theorem}

\begin{proof}
  The invariant three-form \( c \) determines a \( G \)-equivariant
  map \( \Psi\colon M\to Z^3(\g) \), given by
  \begin{equation}
    \label{eq:Psi}
    \inp\Psi{\Lel X \wedge \Lel Y \wedge \Lel Z}
    = c(X,Y,Z)
  \end{equation}
  for \( \Lel X,\Lel Y,\Lel Z \in \g \).  When \( b_2(\g) = 0 =
  b_3(\g) \), for each \( m\in M \) there is a unique element \(
  \nu(m)\in\Pg^* \) satisfying \( \dP\nu(m)=\Psi(m) \).  Since \( \dP
  \) is a \( G \)-morphism, it follows that \( \nu\colon M\to \Pg^* \)
  is also a \( G \)-equivariant.

  We claim that \( \nu \) is a multi-moment map.  Note that, in
  general \( \dP\colon \Pg^* \to Z^3(\g) \cap (\g\wedge\Pg)^* \).  The
  assumption \( b_2(\g) = 0 \), gives that the dual map \( \dP^* \) is
  a surjection \( Z^3(\g)^*\cap(\g\wedge \Pg) \to \Pg \).  This dual
  map is given as minus the adjoint action, since
  \begin{equation}
    \label{eq:adjoint}
    \begin{split}
      &\inp{\dP\alpha}{\Lel Z\wedge \Lel p} = \inp{\dP\alpha}{\Lel
        Z\wedge\sum_{i=1}^k \Lel X_i\wedge \Lel Y_i} \\
      &=-\sum_{i=1}^k (\alpha([\Lel Z,\Lel X_i],\Lel Y_i) +
      \alpha([\Lel X_i,\Lel Y_i],\Lel Z) + \alpha([\Lel Y_i,\Lel
      Z],\Lel X_i)) = -\inp\alpha{\ad_{\Lel Z}(\Lel p)},
    \end{split}
  \end{equation}
  for \( \Lel Z\in \g \), \( \Lel p = \sum_{i=1}^k \Lel X_i\wedge \Lel
  Y_i \in \Pg \).  Hence we may write any \( \Lel p\in\Pg \) in the
  form \( \Lel p = \sum_{i=1}^r \ad_{\Lel Z_i} (\Lel q_i) \), with \(
  \Lel Z_i \in \g \) and \( \Lel q_i \in \Pg \).  Now the function
  \begin{equation*}
    \nu_{\Lel p} = -\sum_{i=1}^r \inp\Psi{\Lel Z_i\wedge \Lel q_i} 
    = -\sum_{i=1}^r c(Z_i\wedge q_i)
  \end{equation*}
  satisfies \( d\nu_{\Lel p} = -\sum_{i=1}^r \Ld_{Z_i}(q_i\hook c) = p
  \hook c \), since \( d(q_i\hook c) = 0 \) by \eqref{eq:kercalc}.
  Moreover we have that
  \begin{equation*}
    \nu_{\Lel p}(m) = -\sum_{i=1}^r \inp{\dP\nu(m)}{\Lel Z_i\wedge
      \Lel q_i} = \sum_{i=1}^r \inp{\nu(m)}{\ad_{\Lel Z_i}(\Lel q_i)} =
    \inp{\nu(m)}{\Lel p}. 
  \end{equation*}
  Thus \( \nu \) is a multi-moment map.

  For the last part of the Theorem, note that a multi-moment map \(
  \nu \) defines elements \( \nu(m) \in \Pg^* \) and the above
  calculations show that \( \dP(\nu(m)) = \Psi(m) \).  However, \(
  b_2(\g) = 0 \) implies that there is at most one solution \( \nu(m)
  \) to this equation, so \( \nu \) is then unique.
\end{proof}

Note that any semi-simple Lie group \( G \) has \( b_1(\g) = 0 =
b_2(\g) \).  Also any reductive group \( G \) with one-dimensional
centre still has \( b_2(\g) = 0 \); in particular this applies to \( G
= \Un(n) \).  So when multi-moment maps for these group actions exist,
they are unique.  However, any simple Lie group \( G \) has \( b_3(\g)
= 1 \), so there can be obstructions to existence.

\section{(2,3)-trivial Lie algebras}
\label{sec:solalg}

In this section we give a structural description of the \(
(2,3)\)-trivial Lie algebras, list them in low dimensions and show
that there are many examples.

\begin{theorem}
  \label{thm:Liealg}
  Any non-trivial finite-dimensional Lie algebra \( \g \ne \bR, \bR^2
  \) satisfying \( b_3(\g)=0 \) is solvable and not nilpotent.  If in
  addition we have that \( b_2(\g)=0 \) then \( \g \) cannot be a
  direct sum of two non-trivial subalgebras, and its derived algebra
  is a codimension one ideal.
\end{theorem}

\begin{proof}
  To verify the first statement, we consider \( \lr \), the solvable
  radical of~\( \g \).  This is the maximal solvable ideal of \( \g \)
  and the quotient \( \g/\lr \) is semi-simple.  By
  \cite{Hochschild-Serre:Cohomology-Lie-algebras}, the cohomology of
  \( \g \) is given by
  \begin{equation*}
    H^k(\g) \cong \sum_{i+j=k} H^i(\g/\lr) \otimes H^j(\lr)^{\g},
  \end{equation*}
  where \( V^{\g} \) is the set of fixed points of the action \( \g \)
  on \( V \).  We thus have \( b_3(\g) \geqslant b_3(\g/\lr) \).  As
  any non-trivial semi-simple Lie algebra has non-trivial third
  cohomology group, we deduce that \( b_3(\g)=0 \) implies \( \g=\lr
  \), so that \( \g \) is solvable.  It is necessarily non-nilpotent
  since it is known \cite{Dixmier:nil-cohomology} that non-Abelian
  nilpotent Lie algebras are of dimension greater than two and have \(
  b_i\geqslant 2 \) for any \( 0 < i < \dim\g \), whereas the only
  non-Abelian three-dimensional nilpotent algebra has \( b_3(\g)=1 \).

  For the second statement of the theorem, suppose \( \g \) is a
  direct sum \( \h\oplus \lk \) of Lie algebras \( \h \) and \( \lk
  \).  Using the Künneth formula, we obtain
  \begin{gather*}
    b_2(\g) = b_2(\h)+b_2(\lk)+b_1(\h)b_1(\lk),\\
    b_3(\g) = b_3(\h)+b_3(\lk)+b_2(\h)b_1(\lk)+b_1(\h)b_2(\lk).
  \end{gather*}
  This immediately gives \( b_2(\h)=0=b_2(\lk) \) and \(
  b_3(\h)=0=b_3(\lk) \).  It also follows that either \( b_1(\h)=0 \)
  or \( b_1(\lk)=0 \).  Reordering the factors, we can assume that \(
  b_1 (\h)=0 \).  Thus \( \h \) has \( b_1(\h) = 0 = b_2(\h) \) and so
  is semi-simple.  But now the number of simple factors of \( \h \) is
  equal to \( b_3(\h) \) which is \( 0 \).  So \( \h = \{0\} \), and
  \( \g \) is not a non-trivial direct sum.

  Now we consider the last assertion of the theorem.  Note that \(
  b_1(\g) = \dim \g - \dim \g' \), where \( \g' = [\g,\g] \) is the
  derived algebra.  As \( \g \) is solvable, we get \( b_1(\g) > 0 \).
  Suppose \( b_1(\g) \geqslant 2 \).  Then there are two linearly
  independent elements \( e_1,e_2 \) in \( Z^1(\g) \).  As \( e_{12}
  := e_1 \wedge e_2 \in Z^2(\g) \) and \( b_2(\g)=0 \), we can find an
  element \( e_3 \) with \( de_3 = e_{12} \).  Note that we have \(
  \dim \Span{e_1,e_2,e_3} = 3 \).  Inductively, we may find \(
  e_4,\dots,e_n \) with \( de_j = e_{1,j-1}\) such that \(
  e_1,\dots,e_n \) is a basis for \( \g \).  But, now \( e_{1n} \in
  Z^2(\g) \) can not be exact, contradicting \( b_2(\g) = 0 \).  Thus,
  we must have \( b_1(\g) = 1 \).
\end{proof}

We will refine this result later, but it is already sufficient to list
the smallest examples of \( (2,3) \)-trivial Lie algebras.  In
dimension one, the only Lie algebra is Abelian and is automatically \(
(2,3) \)-trivial.  In dimension two a Lie algebra is either Abelian or
isomorphic to the \( (2,3) \)-trivial algebra \((0,21)\).  These first
two examples are uninteresting from the point of view of multi-moment
maps since they have \( \Pg=\{0\} \).  However, in dimensions three
and four we may use the known classification of solvable Lie algebras
\cite{Andrada-BDO:four} to obtain more interesting examples.  Note
that for any Lie algebra of dimension \( n \), we have
\begin{equation*}
  \dim\Pg = b_1(\g) + \tfrac12n(n-3),
\end{equation*}
since the kernel of left most map in \eqref{eq:exseq1} is \( H^1(\g) =
Z^1(\g) \).  Thus a \( (2,3) \)-trivial algebra has \( \dim\Pg =
(n-1)(n-2)/2 \), which is non-zero for \( n\geqslant3 \).

\begin{proposition}
  \label{prop:solb34}
  The inequivalent \( (2,3) \)-trivial Lie algebras in dimensions
  three and four are listed in the Tables \ref{tab:3sp} and
  \ref{tab:4sp}.
\end{proposition}

\begin{table}[htp]
  \centering
  \begin{tabular}{LLL}
    \toprule
    \lr_3            & (0,21+31,31)                     &                   \\
    \lr_{3,\lambda}  & (0,21,\lambda.31)                & \lambda \in
    (-1,1]\setminus\{0\} \\
    \lr'_{3,\lambda} & (0,\lambda.21+31,-21+\lambda.31) & \lambda>0         \\
    \bottomrule
  \end{tabular}
  \caption{The inequivalent three-dimensional \((2,3)\)-trivial Lie algebras.}
  \label{tab:3sp}
\end{table}

\begin{table}[htp]
  \centering
  \begin{tabular}{LLL}
    \toprule
    \lr_4&(0,21+31,31+41,41)&\\
    \lr_{4,\lambda}&(0,21,\lambda.31+41,\lambda.41)&\lambda \ne -1,-\tfrac12,0\\ 
    \lr_{4,\mu,\lambda}&(0,21,\mu.31,\lambda.41)&(\mu,\lambda) \in \mathscr R\\
    \lr'_{4,\mu,\lambda}&(0,\mu.21,\lambda.31+41,-31+\lambda.41)&\mu>0,\ 
    \lambda \ne -\tfrac\mu2,0\\
    \ld_{4,\lambda}&(0,\lambda.21,(1-\lambda).31,41+32)&\lambda
    \geqslant \tfrac12,\ \lambda \ne 1,2\\
    \ld'_{4,\lambda}&(0,\lambda.21+31,-21+\lambda.31,2\lambda.41+32)&\lambda> 0\\
    \h_4&(0,21+31,31,2.41+32)&\\
    \bottomrule
  \end{tabular}
  \caption{The inequivalent four-dimensional \( (2,3) \)-trivial Lie
    algebras. The set \( \mathscr R \) consists of the \(
    \mu,\lambda\in (-1,1]\setminus\{0\} \) with \( \lambda\geqslant\mu
    \) and \( \mu+\lambda\ne 0,-1 \).}  
  \label{tab:4sp}
\end{table}

To explain the notation, consider the example \( \h_4 =
(0,21+31,31,2.41+32) \).  This means there is a basis \( e_1,\dots,e_4
\) for \( \h_4^* \) such that \( de_1=0 \), \( de_2 = e_{21}+e_{31}
\), \( de_3 = e_{31} \) and \( de_4 = 2e_{41} + e_{32} \).

We will sketch a proof of this Proposition that is independent of the
classification lists, using the following more detailed structure
result.  Full details of the classification and its extension to
five-dimensional algebras are given in \cite{Madsen-S:2-3-trivial}.

\begin{theorem}
  \label{thm:Hkg}
  A Lie algebra \( \g \) with derived algebra \( \lk=\g' \) is \(
  (2,3) \)-trivial if and only if \( \g \) is solvable, \( \lk \) is
  nilpotent of codimension \( 1 \) in~\( \g \) and \( H^1(\lk)^{\g} =
  \{0\} = H^2(\lk)^{\g} = H^3(\lk)^{\g} \).
\end{theorem}

\begin{proof}
  The derived algebra \( \lk=\g' \) of a solvable algebra \( \g \) is
  always nilpotent, so Theorem~\ref{thm:Liealg} implies that it only
  remains to check the assertions on the \( \g \)-invariant part of
  the cohomology of~\( \lk \).  For this, as \( \lk \) is an ideal
  of~\( \g \), we may use the spectral sequence of Hochschild \& Serre
  \cite{Hochschild-Serre:Cohomology-Lie-algebras} that has \(
  E_2^{j,i} \cong H^j(\g/\lk,H^i(\lk)) \).  Now the codimension one
  condition means that we may write \( \g/\lk = \bR A \) for some
  element~\( A \).  Note that \( H^i(\lk) \) is a \( \g/\lk \) module.
  For any \( \g/\lk \)-module \( M \), the cohomology groups \(
  H^j(\bR A,M) \) are defined from the chain groups \( C^j(\bR A,M) =
  \Lambda^j(\bR A)^*\otimes M = \Hom(\bR A,M) \).  These can only be
  non-zero for \( j=0,1 \) and in both cases they are isomorphic to~\(
  M \).  The chain map is \( d_\bR \) which on \( C^0 \) is \( (d_\bR
  f)(A)= A\cdot f \).  Thus \( E_2^{0,i} = \ker d_\bR = M^A \) and \(
  E_2^{1,1} = M/\img d_\bR \cong \ker d_\bR = M^A \).  We see that the
  \( E_2 \)-term of our spectral sequence is
  \begin{equation*}
    E_2^{j,i} \cong
    \begin{dcases*}
      H^i(\lk)^{\g}&for \( j=0,1 \),\\
      0&otherwise.
    \end{dcases*}
  \end{equation*}
  It follows that the spectral sequence degenerates at the \( E_2
  \)-term and we conclude that
  \begin{gather*}
    H^2(\g) \cong H^2(\lk)^{\g} + H^1(\lk)^{\g},\quad H^3(\g) \cong
    H^3(\lk)^{\g} + H^2(\lk)^{\g},
  \end{gather*}
  from which the result follows.
\end{proof}

\begin{proof}[Sketch proof of Proposition~\ref{prop:solb34}]
  Let \( \g \) be a \( (2,3) \)-trivial algebra of dimension three.
  Then \( \lk=\g' \) is nilpotent and two-dimensional, so \(
  \lk\cong\bR^2 \).  The element \( A \) of Theorem~\ref{thm:Hkg} acts
  on \( \bR^2 \) invertibly and the induced action on \( H^2(\bR^2)
  \cong \Lambda^2\bR^2 \cong \bR \) is also invertible.  So either \(
  A \) is diagonalisable over \( \bC \) with non-zero eigenvalues
  whose sum is non-zero, giving cases \( \lr_{3,\lambda} \) and \(
  \lr'_{3,\lambda} \), or \( A \) acts with Jordan form \( \left(
    \begin{smallmatrix} \lambda&1\\0&\lambda \end{smallmatrix} \right)
  \), \( \lambda\ne0 \), giving case \( \lr_3 \).  The particular
  structure coefficients are obtained by replacing \( A \) by a
  non-zero multiple.

  For \( \g \) of dimension four, we have \( \lk \cong \bR^3 \) or the
  Heisenberg algebra \( \h_3 = (0,0,12) \).  The former gives the
  algebras from the \( \lr \)-series when one enforces that no sum of
  one, two or three eigenvalues of \( A \) is zero.  The latter gives
  the remaining algebras; we have \( H^1(\h_3) \cong \Span{e_1,e_2}
  \), \( H^2(\h_3) \cong \Span{e_{13},e_{23}} \), \( H^3(\h_3) \cong
  \Span{e_{123}} \), \( A \) acts invertibly on these spaces and its
  action in \( e_3 \) is determined by its action on \( e_1 \) and~\(
  e_2 \).
\end{proof}

Theorem~\ref{thm:Hkg} enables us to generate many examples of \( (2,3)
\)-trivial Lie algebras in higher dimensions.  Say that a nilpotent
algebra \( \lk \) is \emph{positively graded} if there is a vector
space direct sum decomposition \( \lk = \lk_1+\dots+\lk_r \) with \(
[\lk_i,\lk_j] \subset \lk_{i+j} \) for all \( i,j \).

\begin{corollary}
  \label{cor:posgrnil}
  Let \( \lk \) be any positively graded nilpotent Lie algebra.  Then
  there is a \( (2,3) \)-trivial Lie algebra whose derived algebra is
  \( \lk \).
\end{corollary}

\begin{proof}
  Let \( \g = \Span A + \lk \) where \( \ad_A \) acts as
  multiplication by \( i \) on \( \lk_i \).  Then \( \g \) is a
  solvable algebra.  Moreover \( (\Lambda^s\lk)^{\g} = \{0\} \) for \(
  s\geqslant 1 \), so the cohomological condition of Theorem
  \ref{thm:Hkg} is satisfied and \( \g \) is as required.
\end{proof}

The algebras constructed in this way are completely solvable, meaning
that each \( \ad_{\Lel X} \), for \( \Lel X \in \g \), has only real
eigenvalues on \( \g \).

\begin{example}
  It may be checked directly that every nilpotent Lie algebra of
  dimension at most six can be positively graded.  The classification
  of these nilpotent algebras (see \cite{Salamon:complex-nil}) then
  gives over \( 30 \) different \( (2,3) \)-trivial algebras in
  dimension~\( 7 \), see~\cite{Madsen-S:2-3-trivial}.
\end{example}

\begin{example}
  Another class of positively graded algebras is given as follows.
  Let \( \Der(\lk) \) be the algebra of derivations of \( \lk \).  A
  \emph{maximal torus} \( \lt \) for \( \lk \) is a maximal Abelian
  subalgebra of the semi-simple elements of \( \Der(\lk) \).  The
  nilpotent Lie algebra \( \lk \) is said to have \emph{maximal rank}
  if \( \dim\lt = \dim(\lk/\lk') \).  Favre~\cite{Favre:poids} showed
  that there are only finitely many systems of weights for such
  algebras and following \cite{Santharoubane:Kac-Moody} a number of
  classification results have been obtained via Kac-Moody techniques,
  see \cite{Fernandez-Ternero:D4} and the references therein.  There
  is a large number (thousands) of families of such algebras.  From
  the general theory, one knows \cite[p.~83]{Favre:poids} that there
  is a positive grading of each maximal rank nilpotent Lie algebra \(
  \lk \).  This grading satisfies \( \sum_{i=s+1}^r \lk_i = \lk^{(s)}
  = [\lk,\lk^{(s-1)}] \).  Thus each of these distinct nilpotent
  algebras of maximal rank arises as the derived algebra of
  non-isomorphic \( (2,3) \)-trivial Lie algebras.
\end{example}

We note that in the construction of Corollary~\ref{cor:posgrnil}, \(
\ad_A \) is a semi-simple derivation of~\( \lk \).  Generally, if \(
\g \) is solvable, then \( A\in \g\setminus \g' \) acts on \( \lk=\g'
\) as a derivation.  For \( \g \) to be \( (2,3) \)-trivial,
Theorem~\ref{thm:Hkg} implies that this action is not nilpotent on \(
H^k(\lk) \) for \( k=1,2,3 \).  For \( \dim\g\geqslant 5 \), this
condition has most force since these three cohomology groups have
dimension at least~\( 2 \) \cite{Dixmier:nil-cohomology}.

Now a nilpotent Lie algebra \( \lk \) is said to be
\emph{characteristically nilpotent} if \( \Der(\lk) \) acts on \( \lk
\) by nilpotent endomorphisms.  It is known that this is equivalent to
\( \Der(\lk) \) being a nilpotent Lie algebra.  For a
characteristically nilpotent algebra \( \lk \), any solvable extension
will act nilpotently on the cohomology of \( \lk \).
Theorem~\ref{thm:Hkg} thus gives the following result.

\begin{corollary}
  If \( \lk \) is a characteristically nilpotent Lie algebra, then \(
  \lk \) is never the derived algebra of a \( (2,3) \)-trivial
  algebra.  \qed
\end{corollary}

\begin{example}
  The first example of a characteristically nilpotent Lie algebra was
  constructed by Dixmier and Lister \cite{Dixmier-L:derivations} in
  dimension eight.  However, there are seven-dimensional examples with
  the same property and even continuous families
  \cite{Goze-K:nilpotent} including:
  \begin{equation*}
    (0,0,12,13,23,14+25+\alpha.23,16+25+35+\alpha.24),\qquad \alpha\ne0.
  \end{equation*}
  Thus no member of this family of algebras can occur as the derived
  algebra of any \( (2,3) \)-trivial Lie algebra.
\end{example}

A Lie algebra \( \g \) is called \emph{unimodular} if the Lie algebra
homomorphism \( \chi\colon\g\to\bR \) given by \( \chi(x) =
\Tr(\ad(x)) \) has trivial image.  Such Lie algebras are interesting
since unimodularity is a necessary condition for the existence of a
co-compact discrete subgroup~\cite{Milnor:left}.

\begin{corollary}
  \label{cor:solb34}
  The simply-connected \( (2,3) \)-trivial Lie groups of dimension
  four or below are not unimodular.  In particular they do not admit a
  compact quotient by a lattice.
\end{corollary}

\begin{proof}
  An \( n \)-dimensional Lie algebra \( \g \) is unimodular if and
  only if \( b_n(\g) = 1 \).  Moreover, one may show that unimodular
  algebras satisfy Hodge duality \( b_k(\g) = b_{n-k}(\g) \).  For \(
  \g \) a \( (2,3) \)-trivial Lie algebra of dimension three, we have
  \( b_3(\g) = 0 \), so \( \g \) is not unimodular.  For \( \g \) of
  dimension four, unimodularity implies \( b_1(\g) = b_3(\g) = 0 \).
  But \( (2,3) \)-trivial algebras have \( b_1(\g) = 1 \), so they can
  not be unimodular in dimension four.
\end{proof}

\begin{example}
  It can be shown that in dimension five and above there are
  unimodular \((2,3)\)-trivial Lie algebras,
  see~\cite{Madsen-S:2-3-trivial}.  Moreover one may verify that there
  are solvmanifolds of the form \( G/\Gamma \), where \( G \) is
  \((2,3) \)-trivial.  Indeed using \cite[Proposition
  7.2.1(i)]{Bock:lowsolv} one may see that there are \( (2,3)
  \)-trivial Lie groups which admit a lattice.  One such example has
  Lie algebra
  \begin{equation*}
    (0,\lambda_1.12,\lambda_2.13,\lambda_3.14,\lambda_4.15),
  \end{equation*}
  where \( \exp(\lambda_i) \approx 0.1277, 0.6297, 2.797, 4.446 \) are
  the four roots of the polynomial \( s^4-8s^3+18s^2-10s+1 \).  As
  this Lie algebra is completely solvable it follows from Hattori's
  Theorem \cite{Hattori:spectral} that one has an isomorphism \(
  H^*_{\text{dR}}\left(G/\Gamma\right)\cong H^*(\g) \).  In particular
  the five-dimensional solvmanifold constructed in this way has
  vanishing second and third de Rham cohomology groups.
\end{example}

\section{Examples and applications}
\label{sec:exaplc}

As strong geometry has no analogue of the Darboux Theorem, the theory
of multi-moment maps is in some senses less rigid than that for
symplectic moment maps and there is a wider variety of types of
example.
	 
\subsection{Second exterior power of the cotangent bundle}
\label{sec:exscdext}

In symplectic geometry one of the fundamental examples is provided by
the cotangent bundle of a manifold, which in mechanics may be
interpreted as a phase space.  In strong geometry, an analogous
example is provided by the second exterior power \( M = \Lambda^2T^*N
\) of a base manifold~\( N \).  This carries a canonical two-form \( b
\), given by
\begin{equation*}
  b_\alpha(W_1,W_2) = \alpha(\pi_*W_1,\pi_*W_2),\qquad\text{\(
    W_1,W_2\in T_\alpha M \),} 
\end{equation*}
where \( \pi\colon \Lambda^2T^*N \to N \) is the bundle projection.
From this one defines a closed three-form \( c \) on~\( M \), via
\begin{equation*}
  c = db.
\end{equation*}
This form is 2-plectic: in local coordinates \( (q^1,\dots,q^n) \) on
\( N \) we have \( \alpha = \sum_{i<j} p_{ij} dq^i\wedge dq^j \)
defining local coordinates \( (q^i,p_{ij}) \) on \( M = \Lambda^2T^*N
\) in which \( c = \sum_{i<j} dp_{ij} \wedge dq^i\wedge dq^j \).  This
is the fundamental example in
\cite{Baez-HR:string,Carinena-CI:multisymplectic}.

If \( G \) is a group of diffeomorphisms of \( N \), then there is an
induced action on \( M = \Lambda^2T^*N \) which preserves \( b \) and
hence \( c \).  As \( c = db \), Proposition~\ref{prop:exact} gives
that there is a multi-moment map \( \nu \) determined
by~\eqref{eq:exmmmap}, which here reads
\begin{equation*}
  \inp{\nu(\alpha)}{\Lel p} = \alpha(p_N)
\end{equation*}
where \( p_N \) is the field of bivectors on \( N \) determined by \(
\Lel p \in \Pg \).  To summarise

\begin{proposition}
  If a Lie group \( G \) acts on a smooth manifold \( N \), then the
  induced action on \( M = \Lambda^2T^*N \) admits a multi-moment map
  with respect to the canonical 2-plectic structure.  \qed
\end{proposition}

\begin{remark}
  Suppose \( N^n \) carries an \( H \)-structure, i.e., a reduction of
  the structure group of \( N \) to \( H\leqslant \GL(n,\bR) \).  Then
  at each point of \( q\in N \) we have a canonical decomposition \(
  \Lambda^2_qT^*N = \oplus_iV_i(q) \) into isotypical \( H \)-modules.
  If the action of \( G \) preserves the \( H \)-structure then the
  induced map of \( \Lambda^2T^*N \) preserves the subbundles \( V_i
  \).  Each bundle \( V_i \) carries a strong geometry via the
  pull-back of \( c \) on \( M=\Lambda^2T^*N \), and the action of \(
  G \) again admits a multi-moment map.  For example, if \( N \) is an
  oriented four-manifold and \( G \) preserves the orientation, then
  there are multi-moment maps \( \nu_\pm \) defined on the 2-plectic
  seven-manifolds \( \Lambda^2_\pm \).  The particular case of \(
  \SO(4)=\Sp(1)_+\Sp(1)_- \) acting on \( N=\bR^4=\bH \) via \(
  (A,B)\cdot q = Aq\overline B \) has multi-moment map on \(
  \Lambda^2_+N \cong \bH + \im\bH \) given by \(
  \inp{\nu_+(q,p)}{a\otimes b} = \tfrac12\re(paqb\overline q) \), for
  \( q\in\bH \), \( p\in\im\bH \), \( a\otimes b \in
  \sP(1)_+\otimes\sP(1)_- = \im\bH\otimes\im\bH \cong \Pg[\sP(1)_+ +
  \sP(1)_-]\).
\end{remark}

\subsection{Homogeneous strong geometries}
\label{sec:orbits}

If \( G \) acts transitively on a strong manifold \( M \), then we may
define \( \Psi\colon M \to Z^3(\g) \) via~\eqref{eq:Psi}, and the
image will be a \( G \)-orbit in \( Z^3(\g) \).  Conversely,
formula~\eqref{eq:Psi} can be used to define strong geometries that
map to a given orbit in \( Z^3(\g) \): given \( \Psi \in Z^3(\g) \),
let be \( K_\Psi \) denote the connected subgroup generated by \(
\ker\Psi = \{\, \Lel X \in \g : \Lel X \hook \Psi = 0 \,\} \); for any
closed group \( H \) of~\( G \) with \( H\subset K_\Psi \),
equation~\eqref{eq:Psi} defines a closed three-form \( c \) on the
homogeneous space \( G/H \) and this strong geometry maps to \(
G\cdot\Psi \subset Z^3(\g) \).

Now suppose that \( \Psi = \dP\beta \) for some \( \beta\in\Pg^* \).
If the map \( \dP \) is injective, then the orbits \( G\cdot\Psi \)
and \( G\cdot\beta \) are identified and the map \( \Psi\colon M\to
Z^3(\g) \) may now be interpreted as a map \( \nu\colon M\to \Pg^* \).
Injectivity of \( \dP \) is guaranteed by the condition \( b_2(\g) = 0
\).  When this holds, the proof of Theorem~\ref{thm:algmmpexc} shows
that \( \nu \) is a multi-moment map for the action of~\( G \).

\begin{theorem}
  \label{thm:orb1}
  Suppose \( G \) is a connected Lie group with \( b_2(\g)=0 \).  Let
  \( \mathcal O = G \cdot \beta \subset \Pg^* \) be an orbit of \( G
  \) acting on the dual of the Lie kernel.  Then there are homogeneous
  strong manifolds \( (G/H,c) \), with \( c \) corresponding to \(
  \Psi = \dP\beta \), such that \( \mathcal O \) is the image of \(
  G/H \) under the (unique) multi-moment map~\( \nu \).

  The strong geometry may be realised on the orbit \( \mathcal O \)
  itself if and only if
  \begin{equation}
    \label{eq:stabeqker}
    \stab_{\g}\beta = \ker(\dP\beta).
  \end{equation}
  In this situation, the orbit is 2-plectic and \( \nu \) is simply
  the inclusion \( \mathcal O \hookrightarrow \Pg^* \).
\end{theorem}

\begin{proof}
  It only remains to prove the assertions of the last paragraph of the
  Theorem.  We have \( \mathcal O = G/K \) with \( K = \stab_G\beta
  \), a closed subgroup of~\( G \).  Now equation~\eqref{eq:adjoint},
  shows that \( K \) has Lie algebra \( \ker(\dP\beta) \), so the
  component of the identity \( K^0 \) of \( K \) is \( K^0 = K_\Psi \)
  for \( \Psi = \dP\beta \).  In particular, \( \Psi \) vanishes on
  elements of \( \lk \) and induces a well-defined form on \(
  T_\beta\mathcal O = \g/\lk \).  The result now follows.
\end{proof}

\begin{example}
  Suppose \( G \) is a \( (2,3) \)-trivial Lie group.  Then, taking \(
  H=\{e\} \), we see that every \( \Psi \in Z^3(\g) \) gives rise to a
  strong geometry on \( G \) with multi-moment map whose image is
  diffeomorphic to the \( G \)-orbit of \( \Psi \).
\end{example}

\begin{example}
  For \( G \) a \( (2,3) \)-trivial group of dimension at most four,
  the Lie kernel contains strong orbits exactly when \( \g' = \h_3 \).
  In this case, \( \Pg \) has dimension~\( 3 \), the orbits are open
  and the strong structure is a left-invariant volume form.
\end{example}

\begin{example}
  Consider \( G = \Un(2) \cong (S^1\times \SU(2))/\{\pm(1,1)\} \).  We
  have \( \Pg[\un(2)] = \Lel T\wedge \su(2) \), where \( \Lel T \)
  generates the Lie algebra of \( S^1 \).  The orbits of \(
  \Pg[\un(2)] \) are thus two-dimensional and can not admit
  (non-trivial) strong geometries.  On the other hand, suppose we
  write \( e_1,e_2,e_3 \) for a standard basis of \( \su(2)^* \) with
  \( de_1 = -e_{23} \).  Then the element \( \beta = dt\wedge e_1 \in
  \Pg[\un(2)]^* \), has \( \dP\beta = - dt \wedge e_{23} \), defining
  \( \Psi \in Z^3(\un(2)) \).  This \( \beta \) does not satisfy
  condition~\eqref{eq:stabeqker} even though \( \dP \) identifies the
  orbits of \( \beta \) and \( \Psi \).  However, \( \Psi \)~defines
  strong geometries on \( \Un(2) \) and on \(
  \Un(2)/\diag(e^{i\theta},e^{-i\theta}) \cong S^1 \times S^2 \) with
  multi-moment map the projection to \( S^2 \).  Note that \(
  \nu\colon \Un(2)\to S^2 \) is essentially the Hopf fibration.
\end{example}

\begin{example}
  Consider \( \g = \su(3) \) as a Lie algebra of complex matrices.
  Write \( E_{pq} \) for the elementary \( 3\times 3 \)-matrix with \(
  1 \) at position \( (p,q) \).  Then \( \su(3) \) has a basis \( A_j
  = i(E_{jj}-E_{j+1,j+1}) \), \( B_{k\ell} = E_{k\ell} - E_{\ell k}
  \), \( C_{k\ell} = i(E_{k\ell} + E_{\ell k}) \), for \( j,k=1,2 \),
  \( k<\ell = 2,3 \).  Let \( a_1,a_2,b_{12},\dots,c_{23} \) denote
  the dual basis.

  The element \( \beta_1 = b_{12} \wedge b_{13} - c_{12} \wedge c_{13}
  \) lies in \( \Pg[\su(3)]^* \).  One has \( \dP\beta_1 =
  3a_1(b_{12}c_{13}-b_{13}c_{12}) \), where we have omitted wedge
  signs.  Direct calculation shows that \( \ker\dP\beta_1 =
  \Span{A_2,B_{23},C_{23}} = \stab_{\su(3)}\beta_1 \).  Thus, by
  Theorem~\ref{thm:orb1}, the \( \SU(3) \)-orbit \( \mathcal O_1 \) of
  \( \beta_1 \) is 2-plectic with multi-moment map given by the
  inclusion in \( \Pg[\su(3)]^* \).  As the above stabiliser is
  isomorphic to \( \su(2) \), we see that up to finite covers \(
  \mathcal O_1 \) is \( \SU(3)/\SU(2) = S^5 \).

  Similarly, one may realise \( F_{1,2}(\bC^3) = \SU(3)/T^2 \) as a
  2-plectic manifold by considering the orbit of \( \beta_2 =
  c_{12}b_{12} + b_{13}c_{13} + c_{23}b_{23} \in \Pg[\su(3)]^* \).

  It is interesting to note that \( F_{1,2}(\bC^3) \) carries a nearly
  Kähler structure.  Such a geometry may be specified by a two-form \(
  \sigma \) and a three-form \( \psi_+ \) whose pointwise stabiliser
  in \( \GL(6,\bR) \) is isomorphic to \( \SU(3) \).  The nearly
  Kähler condition is then \( d\sigma = \psi_+ \), \( d\psi_- =
  -\tfrac12\sigma^2 \), where \( \psi_++i\psi_- \in \Lambda^{3,0} \).
  Direct check shows that each homogeneous strict nearly Kähler
  six-manifold \( G/H = F_{1,2}(\bC^3) \), \( \CP(3) \), \( S^3\times
  S^3 \) and \( S^6 \), as classified by
  Butruille~\cite{Butruille:nK}, may be realised as a 2-plectic orbit
  \( G\cdot\beta \) in~\( \Pg^* \).  Moreover this may done in such a
  way that \( \Psi = \dP\beta \) induces \( c = \psi_+ \) via
  \eqref{eq:Psi} and \( \beta \) induces \( \sigma \) in a
  corresponding way.  Further details may be found
  in~\cite{Madsen-S:2-3-trivial}.
\end{example}

To characterise the homogeneous geometries of Theorem~\ref{thm:orb1},
we introduce the following terminology.

\begin{definition}
  Let \( G \) be a group of symmetries of a strong geometry \( (M,c)
  \).  We say that the action is \emph{weakly \( \Pg \)-transitive} if
  \( G \) acts transitively on~\( M \) and for each non-zero \( X \in
  T_xM \), there is a \( \Lel p\in\Pg \) such that \( c(X\wedge p) \)
  is non-zero.
\end{definition}

\begin{corollary}
  If \( G \) is \( (2,3) \)-trivial, then the weakly \( \Pg
  \)-transitive 2-plectic geometries with symmetry group \( G \) are
  discrete covers of orbits \( \mathcal O = G\cdot\beta \) in~\( \Pg^*
  \) satisfying condition~\eqref{eq:stabeqker}.

  More generally, if \( G \) is a Lie group with \( b_2(\g) = 0 \),
  then the orbits \( \mathcal O = G\cdot\beta \subset \Pg^* \)
  satisfying~\eqref{eq:stabeqker} are, up to discrete covers, the
  weakly \( \Pg \)-transitive 2-plectic geometries that admit a
  multi-moment map.
\end{corollary}

\begin{proof}
  The differential \( \nu_*\colon T_xM \to \Pg^* \) of the
  multi-moment map is given by \( \inp{\nu_*(X)}{\Lel p} = (X\hook
  c)(p) \).  As \( G \) acts weakly \( \Pg \)-transitively, we see
  that \( \nu_*(X) \) is non-zero for each non-zero \( X \).  Thus \(
  \nu_* \) is injective and \( \nu \) has discrete fibres.  Its image
  is an orbit \( G\cdot \beta \) and the proof of
  Theorem~\ref{thm:algmmpexc} shows that the \( 3 \)-form~\( c \) on
  \( M \) is induced by \( \Psi = \dP\beta \).  As \( \nu \) is a
  local diffeomorphism and \( c \)~is 2-plectic it follows that
  \eqref{eq:stabeqker} is satisfied.  Conversely, any orbit \(
  \mathcal O = G\cdot\beta \) satisfying~\eqref{eq:stabeqker} is
  2-plectic with injective multi-moment map~\( \nu \).  Since \( \nu_*
  \) is injective, the equation \( \inp{\nu_*(X)}{\Lel p} = c(X\wedge
  p) \) shows that the action is weakly \( \Pg \)-transitive.
\end{proof}

\subsection{Compact Lie groups with bi-invariant metric}
\label{sec:biinv}

Let \( G \) be a compact semi-simple Lie group.  Its Lie algebra \( \g
\) admits an inner product~\( \inp\cdot\cdot \) invariant under the
adjoint representation, which is proportional to minus the Killing
form.  The left- and right-invariant Cartan one-forms \(
\theta^L,\theta^R\in\Omega^1(G,\g) \) are given by \(
\theta^L(X)=(L_{g^{-1}})_*(X) \), \( \theta^R(X)=(R_{g^{-1}})_*(X) \),
where \( L_g,R_g\colon G\to G \) denote left- and right-multiplication
by~\( g \).  A bi-invariant, and hence closed, three-form is defined
on \( G \) by
\begin{equation}
  \label{eq:bi-invariant}
  c(X,Y,Z) = \inp{[\theta^L(X),\theta^L(Y)]}{\theta^L(Z)},
  \qquad\text{for \( X,Y,Z\in\Gamma(TG) \).}
\end{equation}
This is 2-plectic but is zero on elements of \( \Pg \) for \( G \)
acting on the left.  Instead for \( H,K\leqslant G \), let \( H\times
K \) act on \( G \) by
\begin{equation*}
  (h,k)\cdot g=L_h\circ R_{k^{-1}}(g) = hgk^{-1}.
\end{equation*}
An element \( \Lel X=(\Lel X^H,\Lel X^K)\in\h\oplus\lk \) induces a
vector field \( X \) on \( G \) given by \( X_g = \frac d{dt}
\exp(t\Lel X^H)g\exp(-t\Lel X^K)|_{t=0} = (R_g)_*\Lel X^H -
(L_g)_*\Lel X^K \).  For \( \Lel p = \sum_{j=1}^k\Lel X_j \wedge \Lel
Y_j \in\Pg[\h \oplus \lk] \), we have that \( \sum_{j=1}^k
[X_j^H,Y_j^H] = 0 \) and \( \sum_{j=1}^k [X_j^K,Y_j^K] = 0 \), and
claim that
\begin{equation*}
  \inp{\nu(g)}{\Lel p} = \sum_{j=1}^k \bigl( \inp{\Lel
    X_j^H}{\Ad_g(\Lel Y_j^K)} - \inp{\Lel Y_j^H}{\Ad_g(\Lel X_j^K)}
  \bigr),
\end{equation*}
defines a multi-moment map \( \nu\colon G\to\Pg[\h \oplus \lk]^* \).
This follows from the following computation for \( A_g = (R_g)_*\Lel A
\):
\begin{equation*}
  \begin{split}
    d\inp\nu{\Lel p}(A)_g
    &= \left.\frac d{dt}\inp{\nu(\exp(t\Lel A)g)}{\Lel p}\right|_{t=0}\\
    &= \inp{\Lel X_j^H}{[\Lel A,\Ad_g(\Lel Y_j^K)]} - \inp{\Lel
      Y_j^H}{[\Lel A,\Ad_g(\Lel X^K_j)]}\\
    &=-\inp{[\Ad_{g^{-1}}\Lel X_j^H,\Lel Y_j^K]+[\Lel
      X_j^K,\Ad_{g^{-1}}\Lel Y_j^H]}{\theta^L(A)_g} = (p\hook c)(A)_g,
  \end{split}
\end{equation*}
since \( \theta^L(A)_g = \Ad_{g^{-1}}\Lel A \).  By considering \(
\Lel p \in \Pg[\h \oplus \lk] \) of the form \( \Lel p = (\Lel X^H,0)
\wedge (0,\Lel Y^K) \) with \( \Lel X^H \in \h \) and \( \Lel
Y^K\in\lk \) arbitrary, one finds that
\begin{equation*}
  \ker(\nu_*)_g = (L_g)_*[\Ad_{g^{-1}}\h,\lk]^\perp.
\end{equation*}
In the case that \( \h = \g \), the set \( \ker(\nu_*)_e \) is a
subalgebra of \( \g \) and the image of \( \nu \) is an orbit.

One example is given by \( \h = \g = \su(3) \) and \( \lk = \un(1) =
\diag(ia,-ia,0) \).  Then \( \ker(\nu_*)_e = \un(2) \) and the
multi-moment map~\( \nu \) is the projection from \( \SU(3) \) to \(
\CP(2) = \SU(3)/\Un(2) \).  Now \( \CP(2) \) is quaternionic Kähler,
and \( \SU(3) \) carries a hypercomplex structure
\cite{Joyce:hypercomplex}.  The bi-invariant metric on \( \SU(3) \)
realises the hypercomplex structure as a strong \HKT manifold whose
torsion-three form \( c \) is given by~\eqref{eq:bi-invariant}
\cite{Grantcharov-P:HKT}.  The symmetry group of this \HKT structure
is precisely \( H\times K = \SU(3)\times\Un(1) \) and the map \( \nu
\) realises \( \SU(3) \) as a twisted associated bundle over \( \CP(2)
\) \cite{Pedersen-PS:hypercomplex}.

\subsection{Strong geometries from symplectic manifolds}

Let us show how the theory of multi-moment maps for strong geometries
subsumes that of symplectic moment maps.  Given a symplectic manifold
\( (N,\omega) \) one has a strong geometry on \( M=S^1\times N \) with
\( c = \phi\wedge \omega \), where \( \phi \) is the invariant
one-form dual to the circle action on \( S^1 \).  This geometry is
2-plectic.  If \( N \) comes with a symplectic action of a Lie group
\( H \), then \( G = S^1\times H \) is a symmetry group for the strong
geometry on~\( M \).  The corresponding Lie kernel is given by
\begin{equation*}
  \Pg[\bR+\h] \cong \Pg[\h] + \bR\otimes\h.
\end{equation*}

\begin{proposition}
  \label{thm:sympstrong}
  Let \( (N,\omega) \) be a symplectic manifold with a Hamiltonian
  action of~\( H \), moment map \( \mu\colon N\to \h^* \).  Then \( M
  = S^1 \times N \) carries a strong geometry with symmetry group \( G
  = S^1 \times H\) and this has a multi-moment map~\( \nu \) that may
  be identified with~\( \mu \).
\end{proposition}

\begin{proof}
  We first claim that \( p\hook\omega = 0 \), for each \( \Lel p \in
  \Pg[\h] \subset \Pg \).  Writing \( \Lel p=\sum_{j=1}^k \Lel
  X_j\wedge \Lel Y_j\in\Pg[\h] \), we have
  \begin{equation*}
    \omega(p)
    = \sum_{j=1}^k \omega(X_j,Y_j) = \sum_{j=1}^k Y_j\hook
    d\inp\mu{\Lel X_j}
    =\sum_{j=1}^k \Ld_{Y_j}\inp\mu{\Lel X_j}.
  \end{equation*}
  But \( \mu \) is equivariant, so \( \Ld_Y\inp\mu{\Lel X} =
  \inp\mu{[X,Y]} \).  As \( \sum_{j=1}^k [\Lel X_j,\Lel Y_j] = 0 \) it
  follows that \( \omega(p) = 0 \), as claimed.

  Now we may define \( \nu\colon M\to \Pg^* \) by
  \begin{equation*}
    \inp\nu{\Lel p} = 0,\qquad \inp\nu{\Lel T\wedge \Lel X} =
    \inp\mu{\Lel X},
  \end{equation*}
  for \( \Lel p \in \Pg[\h] \) and \( \Lel X\in \h \), where \( T \)
  is the generator of the \( S^1 \) action on the first factor of \( M
  = S^1 \times G \).  Now \( d\inp\nu{\Lel p} = 0 = p \hook c \) and
  \begin{equation*}
    d\inp\nu{\Lel T\wedge \Lel X} = X\hook\mu = (T\wedge X)\hook c,
  \end{equation*}
  so equation~\eqref{eq:mmmapdefrel} is satisfied.  As the definition
  of \( \nu \) is equivariant, we have that \( \nu \) is a
  multi-moment map.
\end{proof}

\section{Reduction of torsion-free $G_2$-manifolds}
\label{sec:G2}

Let us recall the fundamental aspects of \( G_2 \)-geometry
from~\cite{Bryant:exceptional}.  On \( \bR^7 \) we consider the
three-form \( \phi_0 \) given by
\begin{equation}
  \label{eq:phi0}
  \phi_0 = e_{123} + e_1 (e_{45} + e_{67}) + e_2 (e_{46} -
  e_{57}) - e_3(e_{47} + e_{56}),
\end{equation}
where \( e_1,\dots,e_7 \) is the standard dual basis and wedge signs
have been omitted.  The stabiliser of \( \phi_0 \) is the compact \(
14 \)-dimensional Lie group
\begin{equation*}
  G_2 = \{\, g\in\GL(7,\bR) : g^*\phi_0=\phi_0\,\}.
\end{equation*}
This group preserves the standard metric on \( g_0 =
\sum_{i=1}^7{e_i}^2 \) on \( \bR^7 \) and the volume form \( \vol_0 =
e_{1234567} \).  These tensors are uniquely determined by~\( \phi_0 \)
via the relation \( 6g_0(X,Y)\vol_0 = (X\hook\phi_0) \wedge
(Y\hook\phi_0) \wedge \phi_0 \).  The Hodge \( \Hodge \)-operator
gives a four-form
\begin{equation*}
  \Hodge\phi_0 = e_{4567} + e_{23}(e_{67} + e_{45}) +
  e_{13}(e_{57} - e_{46}) - e_{12}(e_{56} + e_{47}).
\end{equation*}

A \( G_2 \)-structure on a seven-manifold \( Y \) is given by a
three-form \( \phi \in \Omega^3(Y) \) which is linearly equivalent at
each point to \( \phi_0 \).  It determines a metric~\( g \), a volume
form \( \vol \) and a four-form \( \Hodge\phi \) on~\( Y \).  The \(
G_2 \)-structure is called \emph{torsion-free} if both of the forms \(
\phi \) and \( \Hodge\phi \) are closed.  This happens precisely when
\( \LC\phi=0 \) \cite{Fernandez-G:G2}.  One then calls \( (Y,\phi) \)
a torsion-free \( G_2 \)-manifold.  In this situation the metric \( g
\) has holonomy contained in \( G_2 \).

Since a torsion-free \( G_2 \)-geometry comes equipped with a closed
three-form, we may study multi-moment maps for such manifolds.  Let us
assume that \( (Y,\phi) \) has a two-torus symmetry with a
non-constant multi-moment map~\( \nu\colon Y \to \Pg[\bR^2]^* \cong
\bR \).  Choosing generating vector fields \( U \) and \( V \) for the
\( T^2 \)-action, we have \( d\nu = \phi(U,V,\cdot ) \).  The latter
is non-zero if and only if \( U \) and \( V \) are linearly
independent.  So \( T^2 \) acts locally freely on some open set~\( Y_0
\subset Y \).

We may define three two-forms on \( Y_0 \) by
\begin{equation*}
  \omega_0 = V \hook U \hook \Hodge\phi,\quad
  \omega_1 = U \hook \phi
  \quad\text{and}\quad
  \omega_2 = V \hook \phi.
\end{equation*}
To relate these to the \( G_2 \)-structure consider the positive
function \( h \) and one-forms \( \theta_i \) given by
\begin{gather*}
  (g_{UU}g_{VV} - g_{UV}^2)\,h^2 = 1\\
  \theta_1 = h^2(g_{VV} U^\flat - g_{UV} V^\flat),\quad \theta_2 =
  h^2(g_{UU} V^\flat - g_{UV} U^\flat),
\end{gather*}
where \( U^\flat = g(U,\cdot) \) and \( g_{UU} = g(U,U) \), etc.  Note
that \( h \) is well-defined on \( Y_0 \), and that \(
(\theta_1,\theta_2) \) is dual to \( (U,V) \).

\begin{proposition}
  \label{prop:G2T2}
  On \( Y_0 \), the three-form \( \phi \) and the four-form \(
  \Hodge\phi \) are
  \begin{gather*}
    \phi = h^2\omega_0\wedge d\nu + \omega_1\wedge\theta_1 +
    \omega_2\wedge\theta_2 + d\nu\wedge\theta_2\wedge\theta_1,
    \\
    \begin{split}
      \Hodge\phi = \omega_0\wedge\theta_1\wedge\theta_2& + h^2\bigl(
      g_{VV} \omega_1\wedge\theta_2\wedge d\nu - g_{UU}
      \omega_2\wedge\theta_1\wedge d\nu \eqbreak + g_{UV} (
      \omega_1\wedge\theta_1 - \omega_2\wedge\theta_2 ) \wedge d\nu +
      \tfrac12\omega_0\wedge\omega_0 \bigr).
    \end{split}
  \end{gather*}
\end{proposition}

\begin{proof}
  Working locally at a point and using the \( T^2 \)-action we may
  write the first two standard basis elements of \( \bR^7 \) as \( E_1
  = aU = U/g_{UU}^{1/2} \), \( E_2 = {bU + cV} = hg_{UU}^{1/2} (V -
  g_{UV}g_{UU}^{-1} U) \).  We then have \( \theta_1 = ae_1 + be_2 \)
  and \( \theta_2 = ce_2 \).  Now using \eqref{eq:phi0} we get \( ac\,
  d\nu = e_3\), \( ac\, \omega_0 = - (e_{56} + e_{47}) \), \( a\,
  \omega_1 = e_{23} + e_{45} + e_{67} \) and
  \begin{equation*}
    ac\, \omega_2 = -a (e_{13}- e_{46} + e_{57}) - b(e_{23} + e_{45} +
    e_{67}).
  \end{equation*}
  The given expressions now follow.
\end{proof}

Now suppose that \( t\in\nu(Y_0)\subset \bR \) is a regular value for
\( \nu\colon Y_0 \to \bR \).  Then \( \Xt = \nu^{-1}(t) \) is a smooth
hypersurface with unit normal \( N = h(d\nu)^\sharp \).  This inherits
an \( \SU(3) \)-structure \( (\sigma,\psi_\pm) \) given by
\begin{equation}
  \label{eq:SU3-sub}
  \begin{gathered}
    \sigma = N\hook\phi = h\omega_0 +
    h^{-1}\theta_1\wedge\theta_2,\quad \psi_+ = \iota^*\phi =
    \iota^*\omega_1\wedge\theta_1 +
    \iota^*\omega_2\wedge\theta_2,\\
    \begin{aligned}
      \psi_- = -N\hook\Hodge\phi = h\bigl( &g_{VV}
      \iota^*\omega_1\wedge\theta_2 - g_{UU}
      \iota^*\omega_2\wedge\theta_1 \eqbreak + g_{UV} (
      \iota^*\omega_1\wedge \theta_1 - \iota^*\omega_2\wedge\theta_2)
      \bigr),
    \end{aligned}
  \end{gathered}
\end{equation}
where \( \iota\colon X_t \to Y_0 \) is the inclusion.  As shown in
\cite{Chiossi-S:SU3-G2}, oriented hypersurfaces in torsion-free \( G_2
\)-manifolds are \emph{half-flat}, meaning that
\begin{equation}
  \label{eq:half-flat}
  \sigma\wedge d\sigma=0 \quad\text{and}\quad d\psi_+=0.
\end{equation}

Suppose \( T^2 \) acts freely on \( \Xt = \nu^{-1}(t) \).

\begin{definition}
  The \( T^2 \) \emph{reduction} of \( Y \) at level \( t \) is the
  four-manifold
  \begin{equation*}
    M = \nu^{-1}(t)/T^2 = \Xt / T^2.
  \end{equation*}
\end{definition}

\begin{proposition}
  \label{prop:reduction}
  The \( T^2 \) reduction \( M \) carries three pointwise linearly
  independent symplectic forms defining the same orientation.
\end{proposition}

\begin{proof}
  Consider the two-forms \( \omega_0 \), \( \omega_1 \), \( \omega_2
  \) on \( Y_0 \).  These forms are \( T^2 \)-invariant and closed,
  since \( d\omega_0 = \Ld_V(U\hook\Hodge\phi) = 0 \) and \( d\omega_1
  = \Ld_U\phi = 0 \), cf.~\eqref{eq:X-c}.  Furthermore, as \( V\hook
  \omega_1 = d\nu \), their pull-backs to \( \Xt = \nu^{-1}(t) \) are
  basic.  Thus they descend to three closed forms \( \sigma_0 \), \(
  \sigma_1 \) and \( \sigma_2 \) on~\( M \).  The proof of
  Proposition~\ref{prop:G2T2} shows that at a point \( h \sigma_0 =
  -(e_{56} + e_{47}) \), \( h \sigma_1 = c(e_{45} + e_{67}) \) and \(
  h \sigma_2 = a(e_{46}+e_{75}) - b(e_{45}+e_{67}) \), with \( ac = h
  \ne 0 \).  Thus \( \sigma_0 \), \( \sigma_1 \) and \( \sigma_2 \)
  are non-degenerate symplectic forms defining the same orientation.
\end{proof}

The expressions for the forms in this proof show that they satisfy the
following relations on~\( M \):
\begin{equation}
  \label{eq:G2conrel}
  \begin{gathered}
    h^2\,{\sigma_0}^2=g_{UU}^{-1}\,{\sigma_1}^2 = g_{VV}^{-1}\,
    {\sigma_2}^2 = 2\vol_M,\\
    \sigma_0\wedge\sigma_1 = 0 = \sigma_0\wedge\sigma_2,\quad
    \sigma_1\wedge\sigma_2 = 2g_{UV} \vol_M.
  \end{gathered}
\end{equation}
Here \( \vol_M \) is induced by the element \( e_{4567} \) on \( Y \),
which is the volume element on directions orthogonal to the \( T^2
\)-action on~\( \Xt \).  Note that \( (\theta_1,\theta_2) \) is a
connection one-form for \( \Xt \to M \) regarded as a principal \( T^2
\)-bundle.

We now consider how this construction may be inverted, producing the
\( G_2 \)-geometry of \( Y \) from a triple of symplectic forms on a
four-manifold~\( M \).  Note that the relations \eqref{eq:G2conrel}
show that the symplectic forms \( \sigma_i \) define the same
orientation on~\( M \) and are pointwise linearly independent.  Indeed
the intersection matrix \( \tilde Q = (q_{ij}) \) with \(
\sigma_i\wedge\sigma_j = q_{ij}\sigma_0^2 \), for \( i,j=1,2,3 \), is
positive definite.  As in \cite{Donaldson-K:4}, the positive
three-dimensional subbundle \( \Lambda^+ =
\Span{\sigma_0,\sigma_1,\sigma_2} \subset \Lambda^2T^*M \) corresponds
to a unique oriented conformal structure on~\( M \).

\begin{definition}
  A \emph{coherent symplectic triple} \( \CST \) on a four-manifold \(
  M \) consists of three symplectic forms \( \sigma_0 \), \( \sigma_1
  \), \( \sigma_2 \) that pointwise span a maximal positive subspace
  of \( \Lambda^2T^*M \) and satisfy \( \sigma_0 \wedge \sigma_i = 0
  \) for \( i=1,2 \).
\end{definition}

Let \( Q = (q_{ij})_{i,j=1,2} \) be the lower-right \( 2\times 2 \)
submatrix of~\( \tilde Q \).  Since \( \det Q \) is positive, we may
write \( h = \sqrt{\det Q} \in C^\infty(M) \).

\begin{proposition}
  \label{prop:half-flat}
  Let \( (M,\CST) \) be a coherently tri-symplectic four-manifold.
  Suppose \( \X \) is a principal \( T^2 \)-bundle over \( M \) with
  connection one-form \( \Theta = (\theta_1, \theta_2) \).  Then the
  forms \( \sigma \), \( \psi_\pm \) given by
  \begin{equation}
    \label{eq:SU3hfl}
    \begin{gathered}
      \sigma = h\sigma_0 + h^{-1}\theta_1\wedge\theta_2,\quad
      \psi_+ = \sigma_1\wedge\theta_1 + \sigma_2\wedge\theta_2,\\
      \psi_- = h^{-1}(q_{22}\sigma_1\wedge\theta_2 -
      q_{11}\sigma_2\wedge\theta_1 + q_{12}(\sigma_1\wedge\theta_1 -
      \sigma_2\wedge\theta_2))
    \end{gathered}
  \end{equation}
  define an \( \SU(3) \)-structure on \( \X \).  This structure is
  half-flat if and only if \( d\Theta^+ = (\sigma_1,\sigma_2)A \) with
  \( \Tr(AQ) = 0 \).
\end{proposition}

\begin{proof}
  Choose a conformal basis \( e_4,\dots,e_7 \) of \( T^*_xM \) so that
  \( h\sigma_i \) are as in the proof of
  Proposition~\ref{prop:reduction} with \( c^2 = q_{11} \), \( bc =
  -q_{12} \) and \( a^2 = q_{22} - b^2 \).  This is consistent with
  the equation \( ac = h \).  Now inspired by the proof of
  Proposition~\ref{prop:G2T2} we write \( \theta_1 = ae_1 + be_2 \)
  and \( \theta_2 = ce_2 \).  The basis \( e_1,e_2,e_7,e_4,e_6,e_5 \)
  is then an \( \SU(3) \)-basis for \( T^*\X \), with defining forms
  given via~\eqref{eq:SU3-sub} for \( g_{UU} = q_{11}/h^2 \), \(
  g_{UV} = q_{12}/h^2 \) and \( g_{VV} = q_{22}/h^2 \).

  For the final assertion we need to study the
  equations~\eqref{eq:half-flat}.  Firstly, \( \sigma\wedge d\sigma =
  \sigma_0\wedge d\theta_1 \wedge \theta_2 + \sigma_0\wedge d\theta_2
  \wedge \theta_1 \), which vanishes only if \( d\Theta^+ \) is
  orthogonal to \( \sigma_0 \).  This implies that \( d\Theta^+ \) is
  a linear combination \( (\sigma_1,\sigma_2)A \) of \( \sigma_1 \)
  and \( \sigma_2 \).  Now \( d\psi_+ = \sigma_1 \wedge d\theta_1 +
  \sigma_2 \wedge d\theta_2 \), and the vanishing of \( d\psi_+ \)
  gives the constraint \( \Tr(AQ) = 0 \).
\end{proof}

\begin{remark}
  The \( \SU(3) \)-structures found here are more general than those
  studied in~\cite{Goldstein-P:SU3} since the connection one-forms are
  not orthonormal.
\end{remark}

\begin{example}
  Consider \( Y = \bR^7 = \bR \oplus \bC^3 \) endowed with the usual
  three-form and the action of the standard diagonal maximal torus \(
  T^2\subset\SU(3) \).  Concretely, \(\phi\) is given by
  \begin{equation*}
    \phi = \tfrac i2 dx\wedge(dz_1\wedge d\bar z_1+dz_2\wedge
    d\bar z_2+dz_3\wedge d\bar z_3)+\re (dz_1\wedge dz_2\wedge
    dz_3), 
  \end{equation*}
  and \(T^2\) acts by \( (e^{i\theta},e^{i\varphi}) \cdot (x, z_1,
  z_2, z_3) = (x, e^{i\theta}z_1, e^{i\varphi}z_2,
  e^{-i(\theta+\varphi)}z_3) \).  The action is generated by the
  vector fields \( U= \re\{ i(z_1\frac{\partial}{\partial
    z_1}-z_3\frac{\partial}{\partial z_3})\} \) and \( V=\re\{
  i(z_2\frac{\partial}{\partial z_2}-z_3\frac{\partial}{\partial
    z_3})\} \).  It follows that the multi-moment map \(\nu\colon
  Y\to\bR\) is given by
  \begin{equation*}
    \nu(x,z_1,z_2,z_3)=-\tfrac14 \re(z_1z_2z_3).
  \end{equation*} 

  By definition, the \( T^2 \)-reduction of \( Y \) at level \( t \)
  is the quotient space \( M_t=\nu^{-1}(t)/{T^2} \).  In this case
  \(M_0\) is singular, whereas \(M_t\) is a smooth manifold for each
  \( t\ne0 \).  Indeed considering \(\Phi_t \colon M_t \to \bR^4\)
  given by
  \begin{equation*}
    \begin{split}
      \Phi_t(x,z_1,z_2,z_3) &= \bigl(x,
      \tfrac12(\norm{z_1}^2-\norm{z_3}^2),
      \tfrac1{2}(\norm{z_2}^2-\norm{z_3}^2), \im(z_1z_2z_3)\bigr)\\
      &\eqqcolon (x,u,v,w)
    \end{split}
  \end{equation*}
  we have global smooth coordinates on \( M_t \) for \( t\ne0 \).

  In this smooth case, writing \( 4\eta_u = h^2(g_{VV} du - g_{UV} dv)
  \) and \( 4\eta_v = h^2 (g_{UU}dv-g_{UV} du) \), the two-forms
  \(\sigma_0,\,\sigma_1,\,\sigma_2\) are given by
  \begin{gather*}
    4\sigma_0 = dx\wedge dw + dv\wedge du,\quad
    2\sigma_1 = dx\wedge du + dw\wedge \eta_v,\\
    2\sigma_2 = dx\wedge dv + \eta_u \wedge dw .
  \end{gather*}
  These forms depend (implicitly) on \( t \) via the relations \(
  4g_{UU} = \norm{z_1}^2 + \norm{z_3}^2 \), \( 4g_{VV} = \norm{z_2}^2
  + \norm{z_3}^2 \), \( 4g_{UV} = \norm{z_3}^2 \) and \( z_1z_2z_3 =
  -4t + iw \).  In particular, \( g_{UV} \) is a non-constant
  function, so the coherent triple does not specify a hyperKähler a
  structure.  The (oriented) conformal class has representative metric
  \begin{equation*}
    dx^2 + \frac{h^2}{16}dw^2 + 4g_{UU} \eta_u^2 +
    4g_{VV}\eta_v^2 + 4g_{UV} (\eta_u\eta_v + \eta_v\eta_u).
  \end{equation*}
  The curvature of the principal bundle \(\nu^{-1}(t)\to M_t\) is
  given by
  \begin{gather*}
    4d\theta_1 = th^4 dw \wedge ((2g_{VV}-g_{UV})\eta_u +
    (g_{VV}-2g_{UV})\eta_v)\\
    4d\theta_2= th^4 dw \wedge ( (g_{UU}-2g_{UV}) \eta_u +
    (2g_{UU}-g_{UV})\eta_v).
  \end{gather*}

  In the singular case \( t=0 \), the two-torus collapses in two ways:
  to a point along the real axis \( \bR \times \{0\} \subset
  \bR\times\bC^3 \) and to a circle away from \( \bR\times\{0\} \)
  along \( z_1 = z_2 = 0\), \( z_1 = z_3 = 0 \) or \( z_2 = z_3 = 0
  \). The collapsing happens when \( w=0 \) and \( u,v \) satisfy one
  of the following three constraints: \( (u = v \leqslant 0) \), \(
  (u=0,\ v\geqslant 0) \) or \( (u \geqslant 0,\ v=0) \).
\end{example}
 
Studying a certain Hamiltonian flow, Hitchin~\cite{Hitchin:forms}
developed a relationship between torsion-free \( G_2 \)-metrics and
half-flat \( \SU(3) \)-manifolds.  In particular, he derived evolution
equations that describe the one-dimensional flow of a half-flat \(
\SU(3) \)-manifold along its unit normal in a torsion-free \( G_2
\)-manifold.  When the flow equations have a solution, this determines
a torsion-free \( G_2 \)-metric from a half-flat \( \SU(3)
\)-manifold.  In inverting our construction, one could use Hitchin's
flow on the half-flat structure of Proposition~\ref{prop:half-flat}.
However, Hitchin's flow does not preserve the level sets of the
multi-moment map: the unit normal is \( h(d\nu)^\sharp \), but \(
\partial/\partial\nu = h^2(d\nu)^\sharp \).  It is thus more natural
for us to determine the flow equations associated to the latter vector
field.

\begin{proposition}
  Suppose \( T^2 \) acts freely on a connected seven-manifold \( Y \)
  preserving a torsion-free \( G_2 \)-structure \( \phi \) and
  admitting a multi-moment map \( \nu \).  Let \( M \) be the
  topological reduction \( \nu^{-1}(t)/T^2 \) for any \( t \) in the
  image of \( \nu \).  Then \( M \) is equipped with a \( t
  \)-dependent coherent symplectic triple \(
  \sigma_0,\sigma_1,\sigma_2 \) and \( \Xt = \nu^{-1}(t) \) carries
  the half-flat \( \SU(3) \)-structure \( (\sigma,\psi_\pm) \) of
  Proposition~\ref{prop:half-flat}.  The forms on \( \Xt \) satisfy
  the following system of differential equations:
  \begin{equation}
    \label{eq:evol6}
    \begin{gathered}
      \psi_+'=d(h\sigma)\\
      (\tfrac12\sigma^2)'=-d(h\psi_-),
    \end{gathered}
  \end{equation}
  where \('\) denotes differentiation with respect to \( t \).

  Moreover, given a half-flat \( \SU(3) \)-structure on a six-manifold
  \( \Xt[0] \), the system \eqref{eq:evol6} has at most one solution
  and that solution determines a torsion-free \( G_2 \)-structure on
  \( \Xt[0] \times (-\varepsilon,\varepsilon) \) for some \(
  \varepsilon > 0 \).
\end{proposition}

\begin{proof}
  We have
  \begin{equation*}
    \phi=\sigma\wedge hd\nu+\psi_+
    \quad\text{and}\quad \Hodge\phi=\psi_-\wedge hd\nu+\tfrac12\sigma^2.
  \end{equation*}
  These have derivatives
  \begin{gather*}
    d\phi=(hd\sigma+dh\wedge\sigma)\wedge d\nu+d\psi_+,\\
    d\Hodge\phi=(hd\psi_-+dh\wedge\psi_-)\wedge d\nu+\sigma\wedge
    d\sigma
  \end{gather*}
  Half-flatness of \( (\sigma,\psi_\pm) \) gives \( d\phi = 0 =
  d\Hodge\phi \) if and only if
  \begin{equation*}
    0 = \frac{\partial}{\partial\nu} \hook d\phi = -d(h\sigma) +
    \psi'_+ \quad\text{and}\quad 0 = \frac{\partial}{\partial\nu}
    \hook d\Hodge\phi = d(h\psi_-) + \sigma\wedge\sigma'. 
  \end{equation*}
  Hence we have a torsion-free \( G_2 \)-structure if and only if the
  evolution equations \eqref{eq:evol6} are satisfied.

  To demonstrate uniqueness of the solutions we rewrite the evolution
  equations as a complete set of first order differential equations
  for the data on \( M \).  Firstly, the derivatives of \(\sigma_0\),
  \(\sigma_1\), \(\sigma_2\) and \( h \) with respect to
  \({\partial}/{\partial\nu}\) are:
  \begin{equation}
    \begin{gathered}
      \label{eq:flowderM}
      \sigma_0'=0,\quad \sigma_1'=-d\theta_2,\quad \sigma_2'=d\theta_1,\\
      hh'\sigma^2_0=(q_{11}\sigma_2-q_{12}\sigma_1)\wedge
      d\theta_1+(q_{12}\sigma_2-q_{22}\sigma_1)\wedge d\theta_2.
    \end{gathered}
  \end{equation}
  Using \eqref{eq:flowderM} and the definition of \( Q \), we obtain
  the following equations:
  \begin{equation}
    \label{eq:flowderq}
    q_{11}'\sigma_0^2 = -2\sigma_1\wedge d\theta_2,\quad
    q_{22}'\sigma_0^2 = 2\sigma_2\wedge d\theta_1,\quad 
    q_{12}'\sigma_0^2 = \sigma_1\wedge d\theta_1-\sigma_2\wedge
    d\theta_2.
  \end{equation} 
  Finally, combining \eqref{eq:SU3hfl} and \eqref{eq:evol6}, we obtain
  a relation for the derivatives of the connection one-form
  \((\theta_1,\theta_2)\):
  \begin{equation}
    \label{eq:flowdercon}
    \sigma_0\wedge\theta'_1 = dq_{12}\wedge\sigma_2-dq_{22}\wedge\sigma_1,\quad
    \sigma_0\wedge\theta'_2 = dq_{11}\wedge\sigma_2-dq_{12}\wedge\sigma_1.
  \end{equation}
\end{proof}

\begin{remark}
  Modifying the arguments in the proof of \cite[Theorem
  2.3]{Cortes-LSSH:half-flat}, one may verify that the evolution
  equations \eqref{eq:evol6} together with an initial half-flat \(
  \SU(3) \)-structure on \( \Xt[0] \) already ensure that the family
  consists of half-flat structures.  If the initial data are analytic,
  we can solve the flow equations and thereby obtain a holonomy \( G_2
  \)-metric with \( T^2 \)-symmetry.  Indeed, if \( g_M \) is the
  time-dependent metric in the conformal class on \( M \) with volume
  form \( \tfrac12 h^2\sigma_0^2 \), then the \( G_2 \)-metric is
  explicitly
  \begin{equation*}
    h^2dt^2 + g_M + h^{-2}(q_{11} \theta_1^2 + q_{22}\theta_2^2 +
    q_{12}(\theta_1\theta_2+\theta_2\theta_1)). 
  \end{equation*}
  Note that Bryant's study of the Hitchin flow \cite{Bryant:nonemb}
  shows that non-analytic initial data can lead to an ill-posed system
  that has no solution.
\end{remark}

Summarising the results of this section we have:

\begin{theorem}
  Let \( (Y^7,\phi) \) be a torsion-free \( G_2 \)-structure with a
  free \( T^2 \)\bdash symmetry and admitting a multi-moment map.
  Then the reduction \( M \) at a level \( t \) is a coherently
  tri-symplectic four-manifold and the level set \( \Xt \) is a \( T^2
  \)-bundle over \( M \) satisfying the orthogonality condition on \(
  F_+ = d\Theta^+ \) of Proposition~\ref{prop:half-flat}.

  Conversely a coherently tri-symplectic four-manifold together with
  an orthogonal \( F_+ \in \Omega^2(M,\bR^2) \) with integral periods
  define a torsion-free \( G_2 \)-metric with \( T^2 \)-symmetry
  provided the flow equations \eqref{eq:flowderM},
  \eqref{eq:flowderq}, \eqref{eq:flowdercon} admit a solution.  \qed
\end{theorem}

\begin{example}
  Let \( M \) be a hyperK{\"a}hler four-manifold. Then \( M \) comes
  equipped with three symplectic forms \( \sigma_0 \), \( \sigma_1 \),
  \( \sigma_2 \) that satisfy the relations \( \sigma_i\wedge\sigma_j
  = \delta_{ij}\sigma_0^2 \). In particular, \(
  (\sigma_0,\sigma_1,\sigma_2) \) forms a coherent symplectic triple,
  and \( Q \) is the identity matrix: \( h^2 = q^2_{11} = q^2_{22}=1\)
  and \( q_{12}=0 \). If the two-forms \( \sigma_1 \), \( \sigma_2 \)
  have integral periods, we may construct a \( T^2 \)-bundle over \( M
  \) with connection one-form \( \Theta \) that satisfies \( d\Theta =
  (\sigma_1,\sigma_2) \left(\begin{smallmatrix}\alpha & a \\ b &
      -\alpha\end{smallmatrix}\right) \) for integers \(
  \alpha,a,b\in\mathbb{Z} \). The total space \( \mathcal X_0 \) of
  this bundle carries a half-flat \( \SU(3) \)-structure given by
  \eqref{eq:SU3hfl}, and the associated metric is complete if the
  hyperK{\"a}hler base manifold is complete.

  We shall now illustrate how one may solve the flow equations,
  starting from the above data at initial time \( t=0 \). As an a
  priori simplifying assumption, we consider the case when \(
  (d\Theta)'=0 \), i.e., the principal curvatures are \( t
  \)-independent. Then the differential equations for the symplectic
  triple simplify considerably:
  \begin{equation*}
    \sigma_0' = 0,\quad \sigma_1' = -a\Omega_1+\alpha\Omega_2,\quad \sigma_2' = \alpha\Omega_1+b\Omega_2,
  \end{equation*}
  where \(\Omega_1=\sigma_1(0) \),
  \(\Omega_2=\sigma_2(0)\). Integrating these equations, we find that
  \begin{equation*}
    \sigma_0(t) = \sigma_0,\quad \sigma_1(t) = (1-at)\Omega_1+\alpha t\Omega_2,\quad \sigma_2(t) = \alpha t\Omega_1+(1+bt)\Omega_2.
  \end{equation*}
  Using this observation, we may rewrite the equations for \( q'_{ij}
  \) as follows:
  \begin{equation*}
    q'_{11}=2(\alpha^2+a^2)t-2a,\quad q'_{22}=2(\alpha^2+b^2)t+2b, \quad q'_{12}=2\alpha((b-a)t+1),
  \end{equation*}
  and from this we see that \( Q(t)=(1+tA)^2 \), where \( A =
  \left(\begin{smallmatrix} \alpha & a \\ b &
      -\alpha \end{smallmatrix}\right)\left( \begin{smallmatrix}0 & 1
      \\ -1 & 0 \end{smallmatrix}\right) \). As a consequence we have
  that \( dq_{ij}(t)=0 \). Hence, from \eqref{eq:flowdercon}, \(
  \Theta'=0 \) so that \( \Theta(t)=\Theta \). Moreover, one may check
  that the function \( h(t) = \det(A)t^2+\Tr(A)t+1 \) evolves in
  accordance with the equation \(
  hh'\sigma_0^2=(q_{11}\sigma_2-q_{12}\sigma_1)\wedge
  d\theta_1+(q_{12}\sigma_2-q_{22}\sigma_1)\wedge d\theta_2 \).

  The above solution is defined on \( \mathcal X_0 \times I \), where
  the interval \( I\subset\bR \) is determined by non-degeneracy of
  the matrix \( 1+tA \) and \( 0\in I \). By uniqueness of the
  solution on \( \mathcal X_0\times I \), we deduce that the property
  \( (d\Theta)'=0 \) is already implied by the initial data, i.e., it
  is not a simplifying assumption.

  The associated torsion-free \( G_2 \)-structure is determined by the
  three-form
  \begin{equation*}
    \phi=h(t)^2\sigma_0\wedge dt+\theta_1\wedge\theta_2\wedge dt + \sigma_1(t)\wedge\theta_1+\sigma_2(t)\wedge\theta_2,
  \end{equation*}
  and the corresponding holonomy \( G_2 \)-metric is given by
  \begin{equation*}
    g = h(t)^2dt^2+h(t)g_0+h(t)^{-2}(q_{11}(t)\theta_1^2+q_{22}(t)\theta_2^2+q_{12}(t)(\theta_1\theta_2+\theta_2\theta_1)),
  \end{equation*}
  where \( g_0 \) is the initial hyperK{\"a}hler metric on \( M \).

  If the initial hyperK{\"a}hler four-manifold is complete, then we may
  describe completeness properties of \( g \) in terms of the matrix
  \( A \). Provided \( g \) remains finite and non-degenerate, completeness corresponds to completeness of \( h(t)^2dt^2 \) on \( I \), cf. \cite{Bishop-O:warped}. We find that the metric is half-complete,
  cf. \cite{Apostolov-S:K-G2}, precisely when \( \det A\geq 0 \);
  completeness is obtained only for \( A =0 \).
\end{example}

\providecommand{\bysame}{\leavevmode\hbox to3em{\hrulefill}\thinspace}
\providecommand{\MR}{\relax\ifhmode\unskip\space\fi MR }
\providecommand{\MRhref}[2]{%
  \href{http://www.ams.org/mathscinet-getitem?mr=#1}{#2} }
\providecommand{\href}[2]{#2}

\begin{small}
  \parindent0pt\parskip\baselineskip

  T.B.Madsen\textsuperscript{a} \& A.F.Swann\textsuperscript{a,b}

  \textsuperscript{a}Department of Mathematics and Computer Science, University of
  Southern Denmark, Campusvej 55, DK-5230 Odense M, Denmark

  \textit{and}

  CP\textsuperscript3-Origins, Centre of Excellence for Particle
  Physics Phenomenology, University of Southern Denmark, Campusvej 55,
  DK-5230 Odense M, Denmark.

  \textsuperscript{b}Department of Mathematical Sciences, University of
    Aarhus, Ny Munkegade 118, Bldg 1530, DK-8000 Aarhus C, Denmark. 

  \textit{E-mail}: \url{tbmadsen@imada.sdu.dk},
  \url{swann@imada.sdu.dk}, \url{swann@imf.au.dk}
\end{small}

\end{document}